\documentclass[11pt]{article}
\usepackage[margin=2.5cm]{geometry}                
\usepackage{lmodern}
\usepackage{graphicx}
\usepackage{amssymb}
\usepackage{amsmath}
\usepackage{amsthm}
\usepackage{blkarray}
\usepackage{url}
\usepackage{tikz}
\usetikzlibrary{arrows, matrix, arrows.meta, decorations.pathreplacing, decorations.pathmorphing}
\usepackage{bm}
\usepackage{floatpag,mwe}

\makeatletter
\def\paragraph{\@startsection{paragraph}{4}%
  \z@\z@{-\fontdimen2\font}%
  {\normalfont\bfseries}}
\makeatother

\usepackage{hyperref}
\numberwithin{equation}{section}

\newtheorem{thm}{Theorem}[section]
\newtheorem{prop}[thm]{Proposition}
\newtheorem{propdefn}[thm]{Proposition--Definition}
\newtheorem{lem}[thm]{Lemma}
\newtheorem{cor}[thm]{Corollary}

\theoremstyle{definition}
\newtheorem{defn}[thm]{Definition}
\newtheorem{eg}[thm]{Example}
\newtheorem{rmk}[thm]{Remark}

\newcommand{\Aa}{\mathbb{A}}
\newcommand{\CC}{\mathbb{C}}
\newcommand{\FF}{\mathbb{F}}

\newcommand{\PP}{\mathbb{P}}
\newcommand{\QQ}{\mathbb{Q}}

\newcommand{\TT}{\mathbb{T}}
\newcommand{\ZZ}{\mathbb{Z}}

\newcommand{\pA}{\mathsf{A}}

\newcommand{\pC}{\mathsf{C}}

\newcommand{\pG}{\mathsf{G}}
\newcommand{\pT}{\mathsf{T}}

\newcommand{\sA}{\mathcal{A}}
\newcommand{\sB}{\mathcal{B}}

\newcommand{\sO}{\mathcal{O}}

\newcommand{\sX}{\mathcal{X}}

\DeclareMathOperator{\Proj}{Proj}
\DeclareMathOperator{\Spec}{Spec}
\DeclareMathOperator{\Hom}{Hom}

\DeclareMathOperator{\Gr}{Gr}
\DeclareMathOperator{\OGr}{OGr}
\DeclareMathOperator{\Cl}{Cl}
\DeclareMathOperator{\Pf}{Pf}
\DeclareMathOperator{\exc}{exc}
\DeclareMathOperator{\NE}{NE}

\DeclareMathOperator{\Jac}{Jac}
\DeclareMathOperator{\Mor}{Mor}
\DeclareMathOperator{\BDih}{BDih}
\DeclareMathOperator{\Dih}{Dih}

\DeclareMathOperator{\CI}{CI}
\DeclareMathOperator{\codim}{codim}
\DeclareMathOperator{\sing}{sing}

\newcommand{\hidecomments}{0} 

\newcommand{\TomComment}[1]{ 
 \ifnum\hidecomments=1 
 \else
   \textcolor{red}{\em #1}
 \fi}
 
\newcommand{\StephenComment}[1]{
 \ifnum\hidecomments=1 
 \else
   \textcolor{blue}{\em #1}
 \fi}

\begin{document}

\title{Constructing Fano 3-folds from \\ cluster varieties of rank 2}
\author{Stephen Coughlan \& Tom Ducat}
\date{}                      

\maketitle

\begin{center}
\emph{To Miles Reid on his 70th birthday.}
\end{center}

\begin{abstract}
Cluster algebras give rise to a class of Gorenstein rings which enjoy a large amount of symmetry. Concentrating on the rank 2 cases, we show how cluster varieties can be used to construct many interesting projective algebraic varieties. Our main application is then to construct hundreds of families of Fano 3-folds in codimensions 4 and 5. In particular, for Fano 3-folds in codimension 4 we construct at least one family for 187 of the 206 possible Hilbert polynomials contained in the Graded Ring Database.
\end{abstract}

\setcounter{tocdepth}{1}
\tableofcontents

\setlength{\parskip}{2pt}

\section{Introduction}

Cluster algebras were originally introduced in a series of papers by Fomin \& Zelevinsky, starting with \cite{fz}, and have since been found to appear in many diverse branches of mathematics. They enjoy many remarkable properties; two of the most important of which are the \emph{Laurent phenomenon} (i.e.\ that any cluster variable can be expanded as a Laurent polynomial in some distinguished subset of the other cluster variables) and, for cluster algebras of \emph{finite type}, a classification parallel to the Cartan--Killing classification of Lie groups. In particular, a cluster algebra of finite type is generated by a finite number of cluster variables. 

In the language of the wider cluster algebra literature, in this paper we use the term `cluster algebra' to mean a \emph{generalised} cluster algebra $\sA$ with \emph{universal geometric coefficients}, and `cluster variety' to mean the affine variety $X=\Spec\sA$. However as algebraic geometers we like to take a more geometric approach to defining cluster varieties, in terms of the families of log Calabi--Yau surfaces constructed by Gross, Hacking \& Keel \cite{ghk}. Taking this approach provides a much clearer way to generalise our methods.

\subsection{Motivation}

Our primary motivation comes from classification problems in low-dimensional algebraic geometry. In particular, we have chosen to concentrate on the classification of Fano 3-folds (a.k.a.\ $\QQ$-Fano 3-folds with at worst terminal singularities), but the methods of this paper would be just as applicable to constructing other types of projective algebraic varieties, including Calabi--Yau 3-folds, surfaces and 3-folds of general type. Hyperplane sections of our Fano 3-folds are either K3 surfaces or del Pezzo surfaces, with cyclic quotient singularities.

\paragraph{Gorenstein formats.}
For a formal definition of Gorenstein formats and key varieties, we refer to \S\ref{section!cluster-formats} or \cite{bkz}.
Informally, a Gorenstein format is a succinct representation of the generators, relations and syzygies of a Gorenstein ring $R$. A key variety $V$ is the generic case of a format, that is, $V=\Spec R$. We construct $\phi^{-1}(V)\subset\Aa^n_{z_1,\ldots, z_n}$ by substituting the generators $x_1,\dots,x_m$ of $R$ with polynomials $x_i=\phi_i(z_1,\dots,z_n)$, $i=1,\dots,m$. If $\phi$ preserves the format of $\phi^{-1}(V)$, then this is called a regular pullback of $V$. If $R$ is graded and we choose $\phi$ appropriately, then we can divide $\phi^{-1}(V)$ by the $\CC^*$-action to get (weighted) projective varieties. In the best cases, $V$ has a large torus action so that there are several choices of grading available.

For example, the origin $V:=V(x_1,\dots,x_m)$ in $\Aa^m$, together with the Koszul resolution of its defining ideal, is a format. Regular pullbacks of $V$ give hypersurfaces ($m=1$) or complete intersections of codimension $m\ge2$.
Another classic example is the affine cone over the Pl\"ucker embedding of the Grassmannian $\Gr(2,5)$ (cf.~\cite{CR}). This is an instance of the Buchsbaum--Eisenbud theorem for projectively Gorenstein varieties in codimension 3. Moreover, we note that $\Gr(2,5)$ also appears as the simplest nontrivial cluster variety, associated to the $\pA_2$ root system. Brown, Kasprzyk \& Zhu \cite{bkz} make a detailed analysis of $\Gr(2,5)$ format (or, in our notation, $\pA_2$ format) for constructing Calabi--Yau and canonical 3-folds.

Other symmetric spaces, such as the orthogonal Grassmannian $\OGr(5,10)$, were used by Mukai to construct canonical curves, K3 surfaces and smooth Fano 3-folds of genus $6\leq g\leq 10,12$ amongst other things. In particular, it seems that weighted $\OGr(5,10)$ format does not yield any other constructions of Fano 3-folds \cite{CR}, but is moderately successful for canonical 3-folds~\cite{bkz}.

\paragraph{Fano 3-folds.}
Previous efforts to construct Fano 3-folds in codimension~$\geq4$ include Tom \& Jerry \cite{bkr}, \cite{bs}, \cite{bs2}, \cite{prokreid} and \cite{bkq}. 
There are also non-existence results for Fanos of high Fano index due to Prokhorov \cite{P13}, \cite{P16}. 
The approach taken in most of these works is to construct Fano 3-folds by various types of \emph{unprojection}, i.e.\ by starting at the midpoint of a Sarkisov link and working backwards, or something similar. Recently Taylor \cite{rosemary} has developed new types of unprojection to construct many of the codimension 4 Fano 3-fold candidates. 

Despite the geometrically appealing nature of these constructions, 
unfortunately it is difficult to construct birationally rigid varieties this way 
since the corresponding Sarkisov link gives a nontrivial birational map which 
must necessarily return to the variety you started with.
Our cluster formats give uniform descriptions for special subfamilies of the 
Hilbert scheme of Fano 3-folds, with no predisposition to the birational 
geometry. One advantage of this approach is that we construct some examples 
which are expected\footnote{Since the first version of this article appeared, Okada has proven 
birational rigidity in the expected cases \cite{Okada}.} to be birationally rigid 
(see~\S\ref{section!no-projections}).

\paragraph{Similar constructions.}
The $\pC_2$ and $\pG_2$ cluster varieties appearing in this paper have been used in the literature before as key varieties to construct several interesting algebraic varieties. Indeed, $\pC_2$ format appears in the construction of Godeaux surfaces by Reid~\cite{God3} (see also \S\ref{note!Reids-Godeaux}) and a version of $\pG_2$ format appears in the explicit construction of 3-fold flips (in the guise of one of Brown \& Reid's diptych varieties \cite[\S5.2]{dip2}) and 3-fold divisorial contractions \cite[Example~7.2]{phd}.

\subsection{Rank 2 cluster formats}
There are four rank 2 cluster varieties of finite type corresponding to the four rank 2 root systems of finite type: $\pA_1\times\pA_1$, $\pA_2$, $\pC_2$ and $\pG_2$. In each case the cluster algebra has a distinguished set of generators, called cluster variables, which can be put into correspondence with the almost positive roots of the corresponding root system, as shown in Figure~\ref{RootSystemsFig}.\footnote{We fix this as the notation we will use later on, where $\theta_i$ are attached to short roots and $\theta_{ij}$ are attached to long roots.} 
Given two adjacent cluster variables, $\theta_1$ and $\theta_{2}$ say, any other cluster variable $\theta'$ can be written as a Laurent polynomial $\theta' = \tfrac{F(\theta_1,\theta_2)}{\theta_1^{\alpha_1}\theta_2^{\alpha_2}}$ where $\alpha_1 r_1 + \alpha_2 r_2$ is a positive root in the corresponding root system and $r_1,r_2$ are a basis of simple roots.
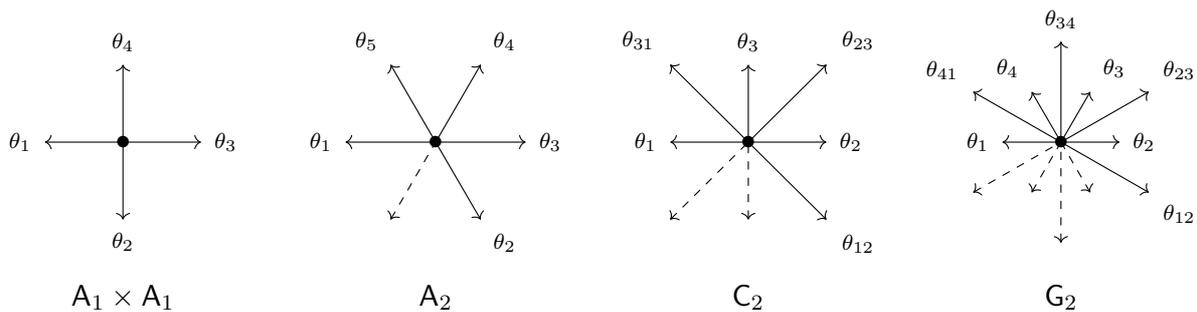
\begin{figure}[h]
\begin{center}
\resizebox{\textwidth}{!}{
\begin{tikzpicture}[label distance=-1mm]
    \node at (1,-2) {$\pA_1\times\pA_1$};
    \draw[<->] (0,0)--(2,0);
    \draw[<->] (1,1)--(1,-1);
    \node at (1,0) {$\bullet$};
    \node at (1,-1) [label={270:\scriptsize $\theta_2$}]{};
    \node at (2,0) [label={0:\scriptsize $\theta_3$}]{};
    \node at (1,1) [label={90:\scriptsize $\theta_4$}]{};
    \node at (0,0) [label={180:\scriptsize $\theta_1$}]{};
    \node at (5,-2) {$\pA_2$};
    \draw[<->] (5 - 2*0.577,0)--(5 + 2*0.577,0);
    \draw[<->] (5 - 0.577,1)--(5 + 0.577,-1);
    \draw[->] (5,0)--(5 + 0.577,1);
    \draw[<-,dashed] (5 - 0.577,-1)--(5,0);
    \node at (5,0) {$\bullet$};
    \node at (5 + 0.577,-1) [label={300:\scriptsize $\theta_2$}]{};
    \node at (5 + 2*0.577,0) [label={0:\scriptsize $\theta_3$}]{};
    \node at (5 + 0.577,1) [label={60:\scriptsize $\theta_4$}]{};
    \node at (5 - 0.577,1) [label={120:\scriptsize $\theta_5$}]{};
    \node at (5 - 2*0.577,0) [label={180:\scriptsize $\theta_1$}]{};
    \node at (9,-2) {$\pC_2$};
    \draw[<->] (8,0)--(10,0);
    \draw[<-,dashed] (9,-1)--(9,0);
    \draw[->] (9,0)--(9,1);
    \draw[<->] (8,1)--(10,-1);
    \draw[<-,dashed] (8,-1)--(9,0);
    \draw[->] (9,0)--(10,1);
    \node at (9,0) {$\bullet$};
    \node at (10,-1) [label={315:\scriptsize $\theta_{12}$}]{};
    \node at (10,0) [label={0:\scriptsize $\theta_2$}]{};
    \node at (10,1) [label={45:\scriptsize $\theta_{23}$}]{};
    \node at (9,1) [label={90:\scriptsize $\theta_3$}]{};
    \node at (8,1) [label={135:\scriptsize $\theta_{31}$}]{};
    \node at (8,0) [label={180:\scriptsize $\theta_1$}]{};
    \node at (13,-2) {$\pG_2$};
    \draw[<->] (13-0.75,0)--(13+0.75,0);
    \draw[<->] (13-0.75*1.5,0.75*0.866)--(13+0.75*1.5,-0.75*0.866);
    \draw[->] (13,0)--(13+0.75*1.5,0.75*0.866);
    \draw[->] (13,0)--(13+0.75*0.5,0.75*0.866);
    \draw[->] (13,0)--(13,1.5*0.866);
    \draw[->] (13,0)--(13-0.75*0.5,0.75*0.866);
    \draw[<-,dashed] (13-0.75*1.5,-0.75*0.866)--(13,0);
    \draw[<-,dashed] (13-0.75*0.5,-0.75*0.866)--(13,0);
    \draw[<-,dashed] (13,-1.5*0.866)--(13,0);
    \draw[<-,dashed] (13+0.75*0.5,-0.75*0.866)--(13,0);
    \node at (13,0) {$\bullet$};
    \node at (13+0.75*1.5,-0.75*0.866) [label={330:\scriptsize $\theta_{12}$}]{};
    \node at (13+0.75,0) [label={0:\scriptsize $\theta_2$}]{};
    \node at (13+0.75*1.5,0.75*0.866) [label={30:\scriptsize $\theta_{23}$}]{};
    \node at (13+0.75*0.5,0.75*0.866) [label={60:\scriptsize $\theta_3$}]{};
    \node at (13,1.5*0.866) [label={90:\scriptsize $\theta_{34}$}]{};
    \node at (13-0.75*0.5,0.75*0.866) [label={120:\scriptsize $\theta_4$}]{};
    \node at (13-0.75*1.5,0.75*0.866) [label={150:\scriptsize $\theta_{41}$}]{};
    \node at (13-0.75,0) [label={180:\scriptsize $\theta_1$}]{};
\end{tikzpicture}}
\caption{The almost positive roots in the root systems of rank 2.}
\label{RootSystemsFig}
\end{center}
\end{figure}

Given three consecutive cluster variables $\theta_{i-1},\theta_i,\theta_{i+1}$ corresponding to roots $r_{i-1},r_i,r_{i+1}$, say, the \emph{tag} at $\theta_i$ is the integer $d_i$ such that $r_{i-1}+r_{i+1}=d_ir_i$. As seen in equations \eqref{eqns!C2-format} and \eqref{eqns!G2-format} below, this tag records the degree of the exchange relation, i.e.\ $\theta_{i-1}\theta_{i+1}=f_i(\theta_i)$ where $f_i$ is a polynomial of degree $d_i$ over an appropriate coefficient ring. 

The simplest rank 2 cluster variety, $\pA_1\times \pA_1$ format, is a generic complete intersection of codimension 2. Moreover, as already mentioned, $\pA_2$ format coincides with $\Gr(2,5)$ format (cf.\ \S\ref{section!A2-format}). In this paper we concentrate on the $\pC_2$ case, which is a Gorenstein format of codimension 4, and the $\pG_2$ case, which is a Gorenstein format of codimension 6. Very concretely, the corresponding cluster varieties are the affine varieties given by the explicit equations appearing below. We will explain one way to derive these equations in \S\ref{section!C2-format} and \S\ref{section!G2-format}, but, for the applications we have in mind, we will essentially use them as black boxes with the nice properties described in \S\ref{subsection!basic-properties}.

\subsubsection{$\pC_2$ format} \label{subsubsection!C2-format}
The cluster variety $X_{\pC_2}=\Spec{\sA_{\pC_2}}\subset \Aa^{13}$ is an affine Gorenstein $9$-fold of codimension 4, where $\sA_{\pC_2}$ is a $\ZZ^6$-graded ring with 13 generators, 9 relations and 16 syzygies. The 13 generators are given by six \emph{cluster variables} $\theta_1$, $\theta_{12}$, $\theta_2$, $\theta_{23}$, $\theta_3$, $\theta_{31}$, six \emph{coefficients} $A_1$, $A_{12}$, $A_2$, $A_{23}$, $A_3$, $A_{31}$ and one \emph{parameter} $\lambda$. The 9 relations are: 
\begin{equation} \label{eqns!C2-format}
\begin{aligned}
\theta_i\theta_j &= A_{ij}\theta_{ij} + A_{jk}A_kA_{ki}  &&&(\times 3) \\
\theta_{ki}\theta_{ij} &= A_i\theta_i^2 + \lambda A_{jk}\theta_i + A_jA_{jk}^2A_k  &&&(\times 3) \\
\theta_i\theta_{jk} &= A_{ij}A_j\theta_j + \lambda A_{ki}A_{ij} + A_kA_{ki}\theta_k  &&&(\times 3)
\end{aligned}
\end{equation}
where $(i,j,k)$ are taken to vary over all $\Dih_6$-permutations of $(1,2,3)$. 

\subsubsection{$\pG_2$ format} \label{subsubsection!G2-format}
The cluster variety $X_{\pG_2}=\Spec{\sA_{\pG_2}}\subset \Aa^{18}$ is an affine Gorenstein $12$-fold of codimension 6, where $\sA_{\pG_2}$ is a $\ZZ^8$-graded ring with 18 generators, 20 relations and 64 syzygies. The 18 generators are given by eight \emph{cluster variables} $\theta_1$, $\theta_{12}$, $\theta_2$, $\theta_{23}$, $\theta_3$, $\theta_{34}$, $\theta_4$, $\theta_{41}$, eight \emph{coefficients} $A_1$, $A_{12}$, $A_2$, $A_{23}$, $A_3$, $A_{34}$, $A_4$, $A_{41}$ and two \emph{parameters} $\lambda_{13}$, $\lambda_{24}$. The 20 relations are: 
\begin{equation} \label{eqns!G2-format}
\begin{aligned}
\theta_i\theta_j &= A_{ij}\theta_{ij} + A_{jk}A_kA_{kl}^2A_lA_{li} &&&(\times 4) \\
\theta_{ij}\theta_{jk} &= A_j\theta_j^3 + \lambda_{jl}A_{kl}A_{li}\theta_j^2 + \lambda_{ik}A_{kl}^2A_lA_{li}^2\theta_j + A_kA_{kl}^3A_l^2A_{li}^3A_i &&&(\times 4) \\
\theta_i\theta_{jk} &= A_{ij}A_j\theta_j^2 + \lambda_{jl}A_{kl}A_{li}A_{ij}\theta_j + \lambda_{ik}A_{kl}^2A_lA_{li}^2A_{ij} + A_kA_{kl}^2A_lA_{li}\theta_k &&&(\times 8) \\
\theta_i\theta_k &= A_{ij}A_jA_{jk}\theta_j + \lambda_{jl}A_{ij}A_{jk}A_{kl}A_{li} + A_{kl}A_lA_{li}\theta_l &&&(\times 2) \\
\theta_{ij}\theta_{kl} &= A_jA_{jk}^2A_k\theta_j\theta_k + A_{li}A_{jk}(\lambda_{ik}\theta_i\theta_k + \lambda_{jl}\theta_j\theta_l) + A_lA_{li}^2A_i\theta_l\theta_i + \cdots\\
 & \cdots + A_{ij}A_{jk}^2A_{kl}A_{li}^2(A_iA_jA_kA_l - \lambda_{ik}\lambda_{jl})  &&&(\times 2)
\end{aligned}
\end{equation}
where $(i,j,k,l)$ are taken to vary over all $\Dih_8$-permutations of $(1,2,3,4)$.

\subsubsection{Relation to Gross, Hacking \& Keel's construction}

Given a positive Looijenga pair $(Y,D)$ (i.e.\ a rational surface $Y$ and an ample anticanonical cycle $D\in|{-K_Y}|$), Gross, Hacking \& Keel \cite{ghk} define a family of mirror surfaces $\sX$ fibred over a toric base variety $B=\Spec\CC[\NE(Y)]$. In this case, the family $\sX/B$ is a relatively Gorenstein affine scheme with nice properties, including a torus grading $\TT^k\curvearrowright\sX$. However we are interested in working with (absolutely) Gorenstein varieties, so instead we consider a slightly different family. We first restrict $\sX$ to $\sX\vert_{\TT^n}$, over the dense torus orbit $\TT^n\subset B$, and then extend this to an affine Gorenstein variety $X/\Aa^n$, corresponding to the closure of $\TT^n\subset\Aa^n$ for some good choice of coordinates on $\TT^n$. We take this $X$ as our rank 2 cluster variety. In particular the theta functions introduced by \cite{ghk} play the role of the cluster variables. Our \emph{coefficients} $A_i$ correspond to coefficients (or frozen variables) in the language of cluster algebras. Our \emph{parameters}, $\lambda$ or $\lambda_{ij}$, do not appear in the original cluster algebra story, however we see them to be unified with the other coefficients by taking this approach.

\subsubsection{Unprojection structure}
These cluster varieties come with a natural Type I Gorenstein projection structure. The champion $X_{\pG_2}$ has a projection to a complete intersection in codimension 2, given by eliminating the four tag 1 cluster variables $\theta_{12}$, $\theta_{23}$, $\theta_{34}$ and $\theta_{41}$. We get a projection cascade (part of which is shown in Figure~\ref{figure!g2-cascade}) in which we see all of the other rank 2 cluster varieties, albeit not in their most natural presentation.
\begin{figure}[h]
\begin{center}
\resizebox{\textwidth}{!}{
\begin{tikzpicture}
  \draw[->] (1.5,0) -- node[above,yshift=0.1cm] {$\widehat{\theta}_{12}$} (2.5,0);
  \draw[->] (5.5,0.8) -- node[above,xshift=-0.2cm,yshift=0.1cm] {$\widehat{\theta}_{34}$} (6.5,1.2);
  \draw[->] (5.5,-0.8) -- node[below,xshift=-0.2cm,yshift=-0.1cm] {$\widehat{\theta}_{23}$} (6.5,-1.2);
  \draw[->] (9.5,1.2) -- node[above,xshift=0.2cm,yshift=0.1cm] {$\widehat{\theta}_{23}$} (10.5,0.8);
  \draw[->] (9.5,-1.2) -- node[below,xshift=0.2cm,yshift=-0.1cm] {$\widehat{\theta}_{34}$} (10.5,-0.8);
  \draw[->] (13.5,0) -- node[above,yshift=0.1cm] {$\widehat{\theta}_{41}$} (14.5,0);

  \draw (1,0) -- (0.707,0.707) -- (0,1) -- (-0.707,0.707) -- (-1,0) -- (-0.707,-0.707) -- (0,-1) -- (0.707,-0.707) -- cycle;
  \node[anchor=180] at (1,0) {\scriptsize $1$};
  \node[anchor=-135] at (0.707,0.707) {\scriptsize $3$};
  \node[anchor=-90] at (0,1) {\scriptsize $1$};
  \node[anchor=-45] at (-0.707,0.707) {\scriptsize $3$};
  \node[anchor=0] at (-1,0) {\scriptsize $1$};
  \node[anchor=45] at (-0.707,-0.707) {\scriptsize $3$};
  \node[anchor=90] at (0,-1) {\scriptsize $1$};
  \node[anchor=135] at (0.707,-0.707) {\scriptsize $3$};
  
  \draw[color=gray!50,fill=gray!50] (4-0.707,0.707) -- (3,0) -- (4-0.707,-0.707) -- cycle;
  \draw (5,0) -- (4.707,0.707) -- (4,1) -- (4-0.707,0.707) -- (4-0.707,-0.707) -- (4,-1) -- (4+0.707,-0.707) -- cycle;
  \node[anchor=180] at (5,0) {\scriptsize $1$};
  \node[anchor=-135] at (4.707,0.707) {\scriptsize $3$};
  \node[anchor=-90] at (4,1) {\scriptsize $1$};
  \node[anchor=-45] at (4-0.707,0.707) {\scriptsize $2$};
  \node[anchor=45] at (4-0.707,-0.707) {\scriptsize $2$};
  \node[anchor=90] at (4,-1) {\scriptsize $1$};
  \node[anchor=135] at (4+0.707,-0.707) {\scriptsize $3$};
  
  \draw[color=gray!50,fill=gray!50] 
     (8-0.707,2.707-0.5) -- (7,2-0.5) -- (8-0.707,2-0.707-0.5) -- cycle
     (8+0.707,2-0.707-0.5) -- (9,2-0.5) -- (8.707,2.707-0.5) -- cycle;
  \draw (8.707,2.707-0.5) -- (8,3-0.5) -- (8-0.707,2.707-0.5) -- (8-0.707,2-0.707-0.5) -- (8,1-0.5) -- (8+0.707,2-0.707-0.5) -- cycle;
  \node[anchor=-135] at (8.707,2.707-0.5) {\scriptsize $2$};
  \node[anchor=-90] at (8,3-0.5) {\scriptsize $1$};
  \node[anchor=-45] at (8-0.707,2.707-0.5) {\scriptsize $2$};
  \node[anchor=45] at (8-0.707,2-0.707-0.5) {\scriptsize $2$};
  \node[anchor=90] at (8,1-0.5) {\scriptsize $1$};
  \node[anchor=135] at (8+0.707,2-0.707-0.5) {\scriptsize $2$};
 
  \draw[color=gray!50,fill=gray!50] 
     (8-0.707,-2+0.707+0.5) -- (7,-2+0.5) -- (8-0.707,-2.707+0.5) -- cycle 
     (8.707,-2+0.707+0.5) -- (8,-1+0.5) -- (8-0.707,-2+0.707+0.5) -- cycle;
  \draw (9,-2+0.5) -- (8.707,-2+0.707+0.5) -- (8-0.707,-2+0.707+0.5) --  (8-0.707,-2.707+0.5) -- (8,-3+0.5) -- (8+0.707,-2.707+0.5) -- cycle;
  \node[anchor=180] at (9,-2+0.5) {\scriptsize $1$};
  \node[anchor=-135] at (8.707,-2+0.707+0.5) {\scriptsize $2$};
  \node[anchor=-45] at (8-0.707,-2+0.707+0.5) {\scriptsize $1$};
  \node[anchor=45] at (8-0.707,-2.707+0.5) {\scriptsize $2$};
  \node[anchor=90] at (8,-3+0.5) {\scriptsize $1$};
  \node[anchor=135] at (8+0.707,-2.707+0.5) {\scriptsize $3$};

  \draw[color=gray!50,fill=gray!50] 
     (12.707,0.707) -- (12,1) -- (12-0.707,0.707) -- cycle
     (12-0.707,0.707) -- (11,0) -- (12-0.707,-0.707) -- cycle
     (12-0.707,-0.707) -- (12,-1) -- (12.707,-0.707) -- cycle;
  \draw[thick] (13,0) -- (12.707,0.707) -- (12-0.707,0.707) -- (12-0.707,-0.707) -- (12+0.707,-0.707) -- cycle;
  \node[anchor=180] at (13,0) {\scriptsize $1$};
  \node[anchor=-135] at (12.707,0.707) {\scriptsize $2$};
  \node[anchor=-45] at (12-0.707,0.707) {\scriptsize $1$};
  \node[anchor=45] at (12-0.707,-0.707) {\scriptsize $1$};
  \node[anchor=135] at (12.707,-0.707) {\scriptsize $2$};

  \draw[color=gray!50,fill=gray!50] 
    (16.707,0.707) -- (16,1) -- (16-0.707,0.707) -- cycle
    (16-0.707,0.707) -- (15,0) -- (16-0.707,-0.707) -- cycle
    (16-0.707,-0.707) -- (16,-1) -- (16.707,-0.707) -- cycle
    (16.707,0.707) -- (17,0) -- (16.707,-0.707) -- cycle;
  \draw[thick] (16.707,0.707) -- (16-0.707,0.707) -- (16-0.707,-0.707) -- (16+0.707,-0.707) -- cycle;
  \node[anchor=-135] at (16.707,0.707) {\scriptsize $1$};
  \node[anchor=-45] at (16-0.707,0.707) {\scriptsize $1$};
  \node[anchor=45] at (16-0.707,-0.707) {\scriptsize $1$};
  \node[anchor=135] at (16.707,-0.707) {\scriptsize $1$};
  
  \node at (0,0) {$\pG_2$};
  \node at (4.1,0) {$\pG_2^{(5)}$};
  \node at (8.1,1.5) {$\pG_2^{(4)}$};
  \node at (8,-1.5) {$\pC_2$};
  \node at (12,0) {$\pA_2$};
  \node at (16.03,0) {$\pA_1\times\pA_1$};
  
  \node at (0,-3.5) {codim $6$};
  \node at (4,-3.5) {codim $5$};
  \node at (8,-3.5) {codim $4$};
  \node at (12,-3.5) {codim $3$};
  \node at (16,-3.5) {codim $2$};
\end{tikzpicture}}
\caption{Part of the projection cascade for the $\pG_2$ cluster variety.}
\label{figure!g2-cascade}
\end{center}
\end{figure}
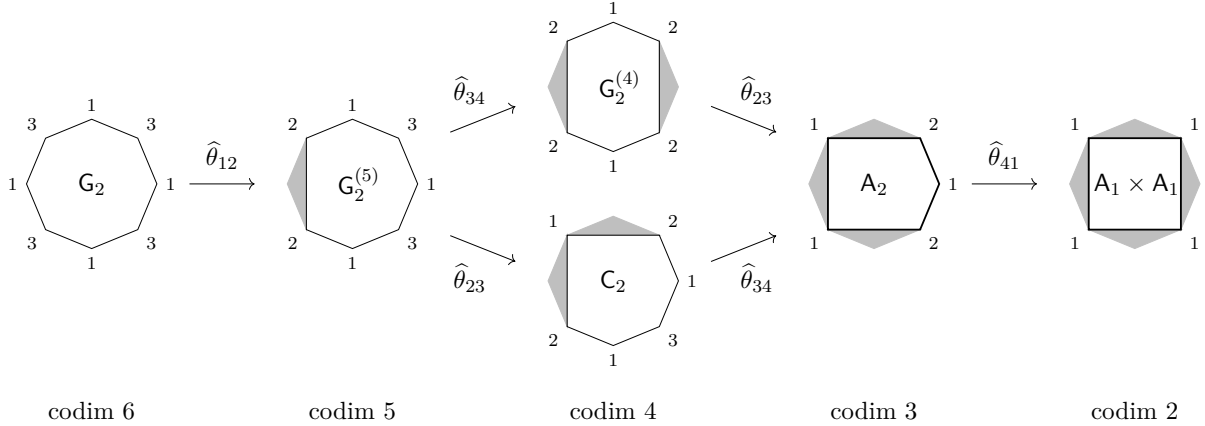
In particular, this projection cascade also allows us to define two intermediate formats: $\pG_2^{(5)}$ and $\pG_2^{(4)}$, where the superscript denotes the codimension. The two codimension 4 formats behave like the two codimension 4 formats Tom \& Jerry~\cite{bkr}. Indeed, $\pC_2$ format can be written as a Tom unprojection from $\pA_2$ format, and $\pG_2^{(4)}$ format as a Jerry unprojection.

This should be an instance of the more general observation that whenever two Looijenga pairs are related by blowing down a $(-1)$-curve in the boundary divisor $\pi\colon(Y',D')\to(Y,D)$ then there is a relationship between the mirror families, described in \cite[\S6.2]{ghk}. The family $\sX/B$ for $(Y,D)$ can be obtained from the family $\sX'/B'$ for $(Y',D')$ as a pullback along the morphism of affine toric varieties $B\to B'$ induced by the inclusion of cones $\pi^*\colon \NE(Y) \to \NE(Y')$.

\subsection{Main results}

For definitions and notation concerning Fano 3-folds, we refer to \S\ref{sec!fanos}. The main result of this paper is the classification and construction of all families of quasismooth Fano 3-folds in $\pC_2$ or $\pG_2$ format. The full classification is available from \cite{bigtables}. In total, we construct over 400 families in codimensions 4 and 5. 
There are none in codimension 6. About two-thirds of these families are prime. The following theorem highlights some more features of the classification.
\begin{thm} \leavevmode
\begin{enumerate}
\item Of the 29 candidates in codimension 4 of index $1$ and with no type I centre, 25 have at least one cluster format construction which is prime;
\item Of the 61 candidates in codimension 4 of index $\ge2$, 45 have at least one cluster format construction which is prime;
\item There are 50 families of codimension 5 Fano 3-folds in a cluster format.
\end{enumerate}
\end{thm}
In particular, when combined with nonexistence results of Prokhorov, part (2) answers the question of existence of Fano 3-folds in codimension 4 with large Fano index.
\begin{cor}
For each candidate Fano 3-fold Hilbert series appearing in \cite{grdb} with codimension 4 and Fano index $q\ge4$, then either there exists a prime Fano 3-fold with that Hilbert series or, by the work of Prokhorov, no such Fano 3-fold exists.
\end{cor}

Several of the codimension 4 candidates have constructions in both $\pC_2$ and $\pG_2^{(4)}$ formats, echoing Tom and Jerry \cite{bkr}. Around 270 of our constructions in codimension 4 have index 1 and a type I centre, and so are special subfamilies of those appearing in \cite{bkr}.

We give a criterion for checking primality in cluster formats. It turns out that the families which are not prime are related to $\PP^2\times\PP^2$, $(\PP^1)^3$ or rolling factors formats. In particular, this answers the question of primality for those cases which overlap with \cite{bkr}.

\subsection{Outline of the paper}

In \S\ref{section!cluster-varieties} we give a brief introduction to cluster varieties, including their important properties. We also give a crash course on Gorenstein formats. In \S\S\ref{section!C2-cluster-format}-\ref{section!G2-cluster-format} we look at the $\pC_2$ and $\pG_2$ rank 2 cluster formats in more detail and explain some ways of constructing them. We make a detailed study of their singular loci and the singular loci of some hyperplane sections, since this plays a crucial part in excluding bad cases from consideration. In \S\ref{sec!fanos} we explain how we apply these formats to construct Fano 3-folds and give many examples. In \S\ref{sec!classification} we explain the computer algorithm that we use to make our classification.

\subsection{Conventions and terminology}
\begin{itemize}
\item Cluster varieties can be defined as schemes over $\ZZ$ but, since the applications we have in mind are constructing complex projective varieties, we choose to work over $\CC$ throughout.
\item We write $\TT^k = (\CC^\times)^k$ for the torus of rank $k$.
\item We write $\Dih_{2n}$ for the dihedral group of order $2n$, which acts on the set $\{1,\ldots,n\}$ labelling the vertices of a regular $n$-gon cyclically. Our cluster formats have variables $\theta_i$, $\theta_{ij}$ etc.\ indexed by $i,j\in\{1,\ldots,n\}$ and an action of $\Dih_{2n}$, where $\pi\in \Dih_{2n}$ acts by $\theta_{ij}\mapsto \theta_{\pi(i)\pi(j)}$ etc. Throughout this paper we always consider our labellings to be \emph{unordered}, e.g.\ $\theta_{ij}=\theta_{ji}$.
\item We write $\CI^{(c)}$ to denote a complete intersection of codimension $c$. 
\item We write down a skew-symmetric matrix $M$ by specifying the strict upper triangular part only. We use $\Pf_k M$ to denote the ideal generated by the $k\times k$ maximal Pfaffians of $M$. 
\item We make free reference to the terminology of Tom \& Jerry \cite{bkr}.
\item A variety $Y$ in weighted projective space is \emph{quasismooth} if the affine cone $\widehat{Y}$ has a worst an isolated singularity at the vertex. 
\item A variable $x$ in a graded ring is \emph{redundant} if it satisfies a relation of the form $x=\cdots$, where $\cdots$ is an expression in terms of the other ring generators.
\end{itemize}

\subsection{Acknowledgements}
It is our pleasure to dedicate this article to our former advisor, Miles Reid, on the occasion of his 70th birthday. Clearly, this work is steeped in his influence, right down to the dilatory write up process. We also thank Gavin Brown, Fabrizio Catanese, Alessio Corti, Paul Hacking, Bal\'azs Szendr\H{o}i, for helpful conversations and encouragement at various stages of this project. SC was supported by ERC Advanced grant no.~340258, TADMICAMT. 


\section{Cluster varieties of rank 2} \label{section!cluster-varieties}

We only give a very brief recap of the theory established by Gross, Hacking \& Keel~\cite{ghk} since we are primarily interested in using our two cluster varieties $X_{\pC_2}$ and $X_{\pG_2}$ to construct examples of Fano $3$-folds. In particular we summarise the results of several calculations without providing many of the details. Hopefully this is enough to provide some motivation for their existence and basic properties, as well as giving some hints as to how other families of log Calabi--Yau surfaces (or higher dimensional varieties) could be used as key varieties. The reader is perfectly entitled to ignore this section if they are willing to take our key varieties $X_{\pC_2}$ and $X_{\pG_2}$ as black boxes with the properties described in \S\ref{subsection!basic-properties} and \S\ref{section!cluster-formats}. 

\subsection{Looijenga pairs}
A \emph{Looijenga pair} $(Y,D)$ is a projective rational surface $Y$ together with a reduced anticanonical cycle $D=\sum_{i=1}^k D_i \in|{-K_{Y}}|$. 

\subsubsection{The $\pA_2$, $\pC_2$ and $\pG_2$ Looijenga pairs} 
We will consider $(Y,D)$ to be one of the following three cases:\footnote{We could also consider the $\pA_1\times\pA_1$ case, with $k=4$ and $(-D_i^2 : i=1,\ldots,4)=(0,0,0,0)$.}
\begin{enumerate}
\item[($\pA_2$)] let $k=5$ and $(-D_i^2 : i=1,\ldots,5)=(1,1,1,1,1)$,
\item[($\pC_2$)] let $k=6$ and $(-D_i^2 : i=1,\ldots,6)=(2,1,2,1,2,1)$,
\item[($\pG_2$)] let $k=8$ and $(-D_i^2 : i=1,\ldots,8)=(3,1,3,1,3,1,3,1)$.
\end{enumerate}
For convenience, in the $\pC_2$ case we relabel the boundary divisors $D_1,D_{12},D_2,\ldots,D_{31}$, so that $D_i^2=-2$ and $D_{ij}^2=-1$, and similarly in the $\pG_2$ case.

\subsubsection{Toric models} \label{sect!toric-models}
Any Looijenga pair can be obtained, possibly after a sequence of toric blowups, as
the blowup of a toric surface $(\bar{Y},\bar{D})$ at points along the toric boundary divisor $\bar{D}$, such that $D\subset Y$ is the strict transform of $\bar{D}\subset\bar{Y}$ (cf.\ \cite[Proposition 1.3]{ghk}). We can realise special cases\footnote{More generally, we could consider blowing up points $e_i$ which lie in general position along $\bar{D}_i\subset \bar{Y}$. However this does not change the final description of our cluster variety.} of the three examples above by considering blowups
\[ \pi_{\pA_2}\colon Y_{\pA_2}\to\PP^2, \qquad \pi_{\pC_2}\colon Y_{\pC_2}\to\PP^2, \qquad \pi_{\pG_2}\colon Y_{\pG_2}\to\PP^1\times\PP^1 \]
at the configurations of points described below, and shown in Figure~\ref{fig!point-configs}. 

Let $\exc(p)$ be the exceptional divisor above a point $p$, let $\pi^{-1}(\bar{C})$ be the strict transform of a curve $\bar C$ under a birational map $\pi$, let $L_{p,q}$ be the line in $\PP^2$ which passes through two points $p,q$, and let $M_{p,q,r}$ be the curve of bidegree $(1,1)$ in $\PP^1\times\PP^1$ which passes through three points $p,q,r$. Then the blowups we consider are given by the following:

\begin{enumerate}
\item[($\pA_2$)] Let $\bar{D}_2+\bar{D}_4+\bar{D}_5$ be the toric boundary components of $\PP^2$. We obtain $Y_{\pA_2}$ by blowing up the two intersection points $d_1=\bar{D}_5\cap \bar{D}_2$, $d_3=\bar{D}_2\cap \bar{D}_4$ and two general points $e_4\in\bar{D}_4$, $e_5\in\bar{D}_5$. 

The anticanonical cycle $D\subset Y_{\pA_2}$ is given by $D_1 = \exc(d_1)$, $D_2 = \pi^{-1}(\bar{D}_2)$, $D_3 = \exc(d_3)$, $D_4 = \pi^{-1}(\bar{D}_4)$ and $D_5 = \pi^{-1}(\bar{D}_5)$. Moreover, we note that $Y_{\pA_2}$ contains five interior $(-1)$-curves $E_1 = \pi^{-1}(L_{d_1,e_4})$, $E_2 = \pi^{-1}(L_{e_4,e_5})$, $E_3 = \pi^{-1}(L_{d_3,e_5})$, $E_4 = \exc(e_4)$ and $E_5 = \exc(e_5)$.

\item[($\pC_2$)] Let $\bar{D}_1+\bar{D}_2+\bar{D}_3$ be the toric boundary components of $\PP^2$. We obtain $Y_{\pC_2}$ by blowing up the three intersection points $d_{ij}=\bar{D}_i\cap \bar{D}_j$ and three points $e_i\in\bar{D}_i\cap \bar{F}$, where $\bar{F}$ is a line in general position with respect to $\bar{D}$.

Let $(i,j,k)$ vary over all $\Dih_6$-permutations of $(1,2,3)$. Then the anticanonical cycle $D\subset Y_{\pA_2}$ is given by $D_i = \pi^{-1}(\bar{D}_i)$ and $D_{ij} = \exc(d_{ij})$. Moreover, we note that $Y_{\pC_2}$ contains six interior $(-1)$-curves $E_i = \exc(e_i)$ and $E_{ij} = \pi^{-1}(L_{d_{ij},e_k})$, and one interior $(-2)$-curve $F=\pi^{-1}(\bar{F})$.

\item[($\pG_2$)] Let $\bar{D}_1+\bar{D}_2+\bar{D}_3+\bar{D}_4$ be the toric boundary components of $\PP^1\times\PP^1$. We obtain $Y_{\pG_2}$ by blowing up the four intersection points $d_{ij}=\bar{D}_i\cap \bar{D}_j$ and four points $e_i\in\bar{D}_i\cap(\bar{F}_{13}\cup\bar{F}_{24})$, where $\bar{F}_{13}$ and $\bar{F}_{24}$ are two curves of bidegree $(1,0)$ and $(0,1)$ which are in general position with respect to $\bar{D}$.

Let $(i,j,k,l)$ vary over all $\Dih_8$-permutations of $(1,2,3,4)$. Then the anticanonical cycle $D\subset Y_{\pG_2}$ is given by $D_i = \pi^{-1}(\bar{D}_i)$ and $D_{ij}=\exc(d_{ij})$. Moreover, we note that $Y_{\pG_2}$ contains eight interior $(-1)$-curves $E_i = \exc(e_i)$ and $E_{ij}=\pi^{-1}(M_{d_{ij},e_k,e_l})$, and two interior $(-2)$-curves $F_{ik} = \pi^{-1}(\bar{F}_{ik})$. 
\end{enumerate}

\begin{figure}[h]
\begin{center}
\begin{tikzpicture}[scale=1]
  \draw[thick] (-1,0) -- (3,0);
  \draw[thick] (3,-1) -- (-1,3);
  \draw[thick] (0,-1) -- (0,3);
  
  \node at (1,-1.5) {$\pA_2$};
  \node at (1,0) [label={270:$e_5$}] {$\bullet$};
  \node at (2,0) [label={above:$d_1$}] {$\bullet$};
  \node at (0,1) [label={180:$e_4$}] {$\bullet$};
  \node at (0,2) [label={right:$d_3$}] {$\bullet$};

  \draw[thick] (4,0) -- (8,0);
  \draw[thick] (5,3) -- (5,-1);
  \draw[thick] (4,3) -- (8,-1);
  \draw[dashed] (4,2) -- (8,-2/5);
  
  \node at (6,-1.5) {$\pC_2$};
  \node at (5,0) [label={225:$d_{31}$}] {$\bullet$};
  \node at (37/5,0) [label={above:$e_3$}] {$\bullet$};
  \node at (7,0) [label={below:$d_{23}$}] {$\bullet$};
  \node at (5,7/5) [label={180:$e_1$}] {$\bullet$};
  \node at (5,2) [label={right:$d_{12}$}] {$\bullet$};
  \node at (6+1/2,1/2) [label={above:$e_2$}] {$\bullet$};
  
  \node at (3.7,2.1) {$\bar{F}$};

  \draw[thick] (9,0) -- (13,0);
  \draw[thick] (12,-1) -- (12,3);
  \draw[thick] (13,2) -- (9,2);
  \draw[thick] (10,-1) -- (10,3);
  \draw[dashed] (11,-1) -- (11,3);
  \draw[dashed] (9,1) -- (13,1);
  
  \node at (11,-1.5) {$\pG_2$};
  \node at (10,0) [label={225:$d_{41}$}] {$\bullet$};
  \node at (11,0) [label={[label distance=-2mm]-45:$e_4$}] {$\bullet$};
  \node at (12,0) [label={315:$d_{34}$}] {$\bullet$};
  \node at (10,1) [label={[label distance=-2mm]-135:$e_1$}] {$\bullet$};
  \node at (12,1) [label={[label distance=-2mm]45:$e_3$}] {$\bullet$};
  \node at (10,2) [label={135:$d_{12}$}] {$\bullet$};
  \node at (11,2) [label={[label distance=-2mm]135:$e_2$}] {$\bullet$};
  \node at (12,2) [label={45:$d_{23}$}] {$\bullet$};
  
  \node at (11,3.3) {$\bar{F}_{24}$};
  \node at (13.3,1) {$\bar{F}_{13}$};
  
\end{tikzpicture}
\caption{The configurations of points blown up to obtain $Y_{\pA_2}$, $Y_{\pC_2}$ and $Y_{\pG_2}$.}
\label{fig!point-configs}
\end{center}
\end{figure}
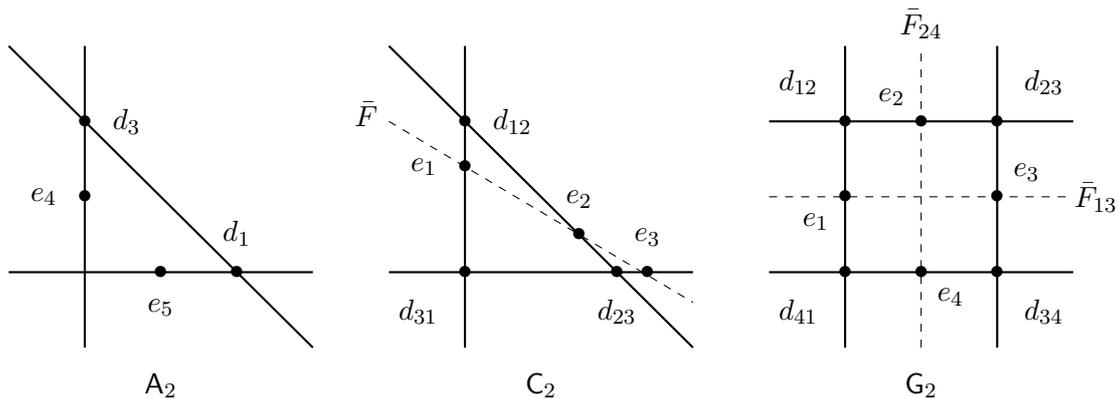

\subsubsection{The Mori cone $\NE(Y)$}
A Looijenga pair $(Y,D)$ is said to be \emph{positive} if the cycle $D$ supports an ample divisor. In particular this implies that the Mori cone $\NE(Y)\subset N_1(Y)_\QQ$ is a closed finite polyhedral cone. 

In the $\pA_2$, $\pC_2$ and $\pG_2$ cases, $(Y,D)$ is positive and the Mori cone $\NE(Y)$ is spanned by $10$, $13$ and $18$ extremal rays respectively, corresponding to the classes $[D_i],[D_{ij}],[E_i],[E_{ij}],[F],[F_{ik}]$ described above. A $(-1)$-curve contained in $Y\setminus D$ must intersect the boundary divisor $D$ in precisely one point, in the interior of a component $D_i$ or $D_{ij}$. In each of the three cases there is precisely one $(-1)$ curve $E_i$ which intersects $D_i$ and one $(-1)$-curve $E_{ij}$ which intersects $D_{ij}$. Figure~\ref{figure!dual-intersection-diagrams} depicts the dual intersection diagrams for the curves in $Y$ belonging to the extremal rays of $\NE(Y)$.\footnote{There are two pairs of double edges in the $\pG_2$ graph since $E_{12}\cdot E_{34}=E_{23}\cdot E_{41}=2$ in $Y_{\pG_2}$, but otherwise all of the (non-self-)intersection numbers are $0$ or $1$.} 
\begin{figure}[h]
\begin{center} 
\resizebox{\textwidth}{!}{
\begin{tikzpicture}  
   \node (D1) at ({4*cos(0*72)},{4*sin(0*72)}) {$D_1$}; 
   \node (D2) at ({4*cos(1*72)},{4*sin(1*72)}) {$D_2$}; 
   \node (D3) at ({4*cos(2*72)},{4*sin(2*72)}) {$D_3$}; 
   \node (D4) at ({4*cos(3*72)},{4*sin(3*72)}) {$D_4$}; 
   \node (D5) at ({4*cos(4*72)},{4*sin(4*72)}) {$D_5$};
   \node (E1) at ({2*cos(0*72)},{2*sin(0*72)}) {$E_1$}; 
   \node (E2) at ({2*cos(1*72)},{2*sin(1*72)}) {$E_2$}; 
   \node (E3) at ({2*cos(2*72)},{2*sin(2*72)}) {$E_3$}; 
   \node (E4) at ({2*cos(3*72)},{2*sin(3*72)}) {$E_4$}; 
   \node (E5) at ({2*cos(4*72)},{2*sin(4*72)}) {$E_5$}; 
   
   \draw (D5) -- (D1) -- (D2) -- (D3) -- (D4) -- (D5) -- (E5) -- (E3) -- (E1) -- (E4) -- (E2) -- (D2);
   \draw (D3) -- (E3) (D4) -- (E4) (D1) -- (E1)  (E2) -- (E5);

   \node (D31)  at ({10+4*cos(0*60)},{4*sin(0*60)}) {$D_1$}; 
   \node (D1) at ({10+4*cos(1*60)},{4*sin(1*60)}) {$D_{12}$}; 
   \node (D12)  at ({10+4*cos(2*60)},{4*sin(2*60)}) {$D_2$}; 
   \node (D2) at ({10+4*cos(3*60)},{4*sin(3*60)}) {$D_{23}$}; 
   \node (D23)  at ({10+4*cos(4*60)},{4*sin(4*60)}) {$D_3$}; 
   \node (D3) at ({10+4*cos(5*60)},{4*sin(5*60)}) {$D_{31}$}; 
   \node (B31)  at ({10+2.5*cos(0*60)},{2.5*sin(0*60)}) {$E_1$}; 
   \node (A1) at ({10+2.5*cos(1*60)},{2.5*sin(1*60)}) {$E_{12}$}; 
   \node (B12)  at ({10+2.5*cos(2*60)},{2.5*sin(2*60)}) {$E_2$}; 
   \node (A2) at ({10+2.5*cos(3*60)},{2.5*sin(3*60)}) {$E_{23}$}; 
   \node (B23)  at ({10+2.5*cos(4*60)},{2.5*sin(4*60)}) {$E_3$}; 
   \node (A3) at ({10+2.5*cos(5*60)},{2.5*sin(5*60)}) {$E_{31}$}; 
   \node (C) at (10,0) {$F$}; 
   
   \draw (D1)--(D12)--(D2)--(D23)--(D3)--(D31)--(D1)--(A1)--(A2)--(A3)--(A1);
   \draw (D2)--(A2) (D3)--(A3);
   \draw (B12)--(C)--(B31) (C)--(B23); 
   \draw (B23) to [out=80,in=220](A1);
   \draw (B31) to [out=200,in=340](A2);
   \draw (B12) to [out=320,in=100](A3);
   \draw (D12)--(B12) (D23)--(B23) (D31)--(B31);
 
   \node (E41)  at ({20+4*cos(0*45)},{4*sin(0*45)}) {$D_1$}; 
   \node (E1) at ({20+4*cos(1*45)},{4*sin(1*45)}) {$D_{12}$}; 
   \node (E12)  at ({20+4*cos(2*45)},{4*sin(2*45)}) {$D_2$}; 
   \node (E2) at ({20+4*cos(3*45)},{4*sin(3*45)}) {$D_{23}$}; 
   \node (E23)  at ({20+4*cos(4*45)},{4*sin(4*45)}) {$D_3$}; 
   \node (E3) at ({20+4*cos(5*45)},{4*sin(5*45)}) {$D_{34}$}; 
   \node (E34)  at ({20+4*cos(6*45)},{4*sin(6*45)}) {$D_4$}; 
   \node (E4) at ({20+4*cos(7*45)},{4*sin(7*45)}) {$D_{41}$}; 
   
   \node (B41)  at ({20+3*cos(0*22.5)},{3*sin(0*22.5)}) {$E_1$};  
   \node (A1) at ({20+3*cos(2*22.5)},{3*sin(2*22.5)}) {$E_{12}$}; 
   \node (B12)  at ({20+3*cos(4*22.5)},{3*sin(4*22.5)}) {$E_2$};  
   \node (A2) at ({20+3*cos(6*22.5)},{3*sin(6*22.5)}) {$E_{23}$};
   \node (B23)  at ({20+3*cos(8*22.5)},{3*sin(8*22.5)}) {$E_3$}; 
   \node (A3) at ({20+3*cos(10*22.5)},{3*sin(10*22.5)}) {$E_{34}$}; 
   \node (B34)  at ({20+3*cos(12*22.5)},{3*sin(12*22.5)}) {$E_4$};  
   \node (A4) at ({20+3*cos(14*22.5)},{3*sin(14*22.5)}) {$E_{41}$}; 
   \node (C1) at (19.5,0.5) {$F_{24}$}; 
   \node (C2) at (20.5,-0.5) {$F_{13}$}; 
   
   \draw (E1)--(E12)--(E2)--(E23)--(E3)--(E34)--(E4)--(E41)--(E1);
   \draw (E1)--(A1) (E2)--(A2) (E3)--(A3) (E4)--(A4);
   \draw (E12)--(B12) (E23)--(B23) (E34)--(B34) (E41)--(B41);
   \draw (C1)--(C2);
   \draw (A1) to [out=190,in=80](A3);
   \draw (A1) to [out=180+80,in=10](A3);
   \draw (A2) to [out=270+10,in=90+80](A4);
   \draw (A2) to [out=270+80,in=90+10](A4);   
   \draw (A1)--(A2)--(A3)--(A4)--(A1);
   \draw (B34)--(A1)--(B23)--(A4)--(B12)--(A3)--(B41)--(A2)--(B34);  
   \draw (B12)--(C1)--(B34) (B23)--(C2)--(B41);
\end{tikzpicture}}
\end{center}
\caption{Dual intersection diagrams for the extremal curves on $Y_{\pA_2}$, $Y_{\pC_2}$ and $Y_{\pG_2}$.}
\label{figure!dual-intersection-diagrams}
\end{figure}
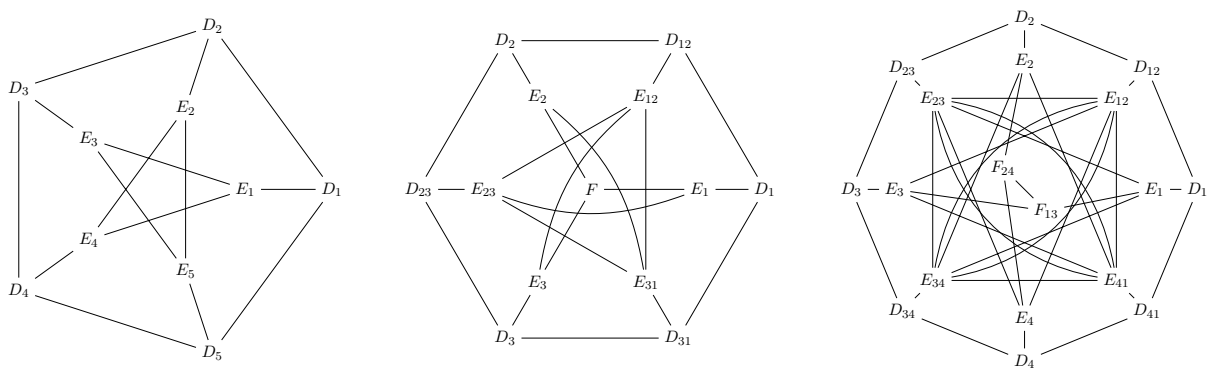

\subsubsection{The intersection pairing and the Looijenga roots}
We have the usual intersection pairing: 
\[ ({\:}\cdot{\:}) \colon N_1(Y)_\QQ\times N^1(Y)_\QQ \to \QQ \]
Let $\bm{D}\subset N^1(Y)_\QQ$ be the sublattice $\bm{D}=\bigoplus_{i=1}^k\ZZ [D_i]$, spanned by the components of $D$. Then elements $\alpha$ of the subspace 
\[ \bm{D}^\perp=\left\{ [C]\in N_1(Y)_\QQ : [D]\cdot[C] = 0 \text{ for all }D\in\bm{D} \right\}\subset N_1(Y)_\QQ \]
satisfying $\alpha^2=-2$ are called \emph{Looijenga roots}.\footnote{In general, for a Looijenga pair $(Y,D)$ whose boundary divisor has a negative definite intersection matrix, there is a further condition to ensure that $\alpha$ corresponds to the parallel transport of the class of an internal $(-2)$-curve on a deformation equivalent pair, cf.\ \cite[Theorem 3.3]{ghk2}.} In the $\pA_2$ case $\bm{D}^\perp=\emptyset$, in the $\pC_2$ case $\bm{D}^\perp=\ZZ[F]$ forms a root system of type $\pA_1$ and in the $\pG_2$ case $\bm{D}^\perp=\ZZ\langle[F_{13}],[F_{24}]\rangle$ forms a root system of type $\pA_2$.

\subsection{The mirror family $\sX$ and the cluster variety $X$}
We now describe the mirror family $\sX$ introduced by Gross, Hacking \& Keel and the related cluster variety $X$. In both cases these are families of mildly singular (log canonical) surfaces. The first $\sX$ is defined over a singular base variety $B$, whereas $X$ is defined over a much nicer base variety $\Aa^n$.

\subsubsection{The mirror family $\sX$} \label{sect!curlyX}
The mirror family $\sX$, for a Looijenga pair $(Y,D)$, is a deformation of the \emph{vertex} 
\[ \mathbb{V}_k = \left(\Aa^2_{\theta_1,\theta_2} \cup \Aa^2_{\theta_2,\theta_3} \cup \cdots \cup \Aa^2_{\theta_{k-1},\theta_k} \cup \Aa^2_{\theta_k,\theta_1} \right) \subseteq \Aa^{k}_{\theta_1,\ldots,\theta_k}, \] 
with $k$ equal to the number of components of $D=\bigcup_{i=1}^kD_i$, defined by introducing theta functions and using the machinery of scattering diagrams. If $(Y,D)$ is positive, then the construction yields an algebraic variety $\sX$ with the following nice properties:
\begin{enumerate}
\item $\sX/B$ is a deformation of $\mathbb{V}_k$ over the affine toric variety:
\[ B := \Spec\left(\CC[\NE(Y)]\right) = \Spec \left(\CC\left[ z^{C} : [C] \in \NE(Y) \right] \right) \]
\item $\sX/B$ is a flat family of affine Gorenstein surfaces with at worst semi-log canonical singularities,
\item the action of the torus $\TT^D = \bm{D}\otimes \CC^*$ on $B$, given by 
\[ \lambda_i\cdot \left(z^{C}\right) = \lambda_i^{D_i\cdot C}z^{C} \qquad \text{for $i=1,\ldots,k$} \]
extends uniquely to a $\TT^D$-action on $\sX$.
\end{enumerate}
Our only problem with trying to use $\sX/B$ as a key variety directly is that the total space $\sX$ is not Gorenstein, but only \emph{relatively} Gorenstein. 

\subsubsection{The cluster variety $X$}
For that reason we consider a slightly different family, by first taking the restriction $\sX|_{\TT^n}$ to the structure torus $\TT^n\subset B$ and then by considering the variety $X/\Aa^n$, obtained by the compactification $\TT^n\subset\Aa^n$ with respect to some natural choice of coordinates on $\TT^n$. 
\begin{center}\begin{tikzpicture}[scale=0.8]
   \node (A) at (0,2) {$\sX$};
   \node (B) at (0,0) {$B$};
   \node (C) at (2,2) {$\sX\vert_{\TT^n}$};
   \node (D) at (2,0) {$\TT^n$};
   \node (E) at (4,2) {$X$};
   \node (F) at (4,0) {$\Aa^n$};
   
   \draw[->] (A) -- (B);
   \draw[->] (C)--(D);
   \draw[->] (E)--(F);
   \draw[left hook->] (D)--(B);
   \draw[right hook->] (D)--(F);
   \draw[left hook->] (C)--(A);
   \draw[right hook->] (C)--(E);
\end{tikzpicture}\end{center} 

This choice of coordinates is described in \S\ref{section!C2-format} for the $\pC_2$ cluster variety and in \S\ref{section!G2-format} for the $\pG_2$ cluster variety.

\subsection{Basic properties}\label{subsection!basic-properties}
The cluster variety $X$ inherits all of the good properties of the mirror family $\sX$. In particular, $X$ is an normal affine Gorenstein variety and has a $\TT^D$ action.
We summarise some of the basic properties of the cluster variety $X$ that will be important later on.

\begin{prop}
The cluster variety $X=\Spec \sA$ has the following properties:
\begin{enumerate}
\item $X$ is a normal, prime, Gorenstein, affine variety,
\item $X$ has an action by $G\times \TT^{D}$ for some finite symmetry group $G$.
\end{enumerate}
\end{prop}




Moreover, because of the nice structure of the equations we also have the following Lemma.

\begin{lem}\label{lem!partial-covering}
The cluster variety $X$ has a partial open covering by complete intersection affine hypersurfaces 
\[  U_i = X \cap( \theta_i \neq 0 ) \quad \text{and} \quad U_{ij} = X \cap( \theta_{ij} \neq 0 ).\]
\end{lem}

The complement of these pieces is called the \emph{deep locus} of $X$ and breaks up into a union of subvarieties of very high codimension. See \S\ref{subsection!C2-affine-pieces} for an example.

\begin{rmk}
These open sets make it possible to check the singular loci (see \S\ref{section!quasismooth-strategy}) and compute the rank of the divisor class group (see \S\ref{sec!primality}) of regular pullbacks from $X$. 
\end{rmk}

\subsection{The $\pA_2$ case} \label{section!A2-format}

As a warm-up we explain how this works in the $\pA_2$ case. 

\subsubsection{Equations for the mirror family $\sX_{\pA_2}$}
The equations for $\sX_{\pA_2}/B_{\pA_2}$ are worked out in \cite[Example 3.7]{ghk}. To simplify the notation we let $A_i=z^{[D_i]}$ and $B_i=z^{[E_i]}$. The base variety $B_{\pA_2}$ is a toric variety defined by 10 equations: 
\[ A_iB_i = A_{i-2}A_{i+2} = B_{i-1}B_{i+1} \]
and there are five relative equations:
\[ \theta_{i-1}\theta_{i+1} = A_i\theta_i + A_iB_i \] 
which define $\sX_{\pA_2}$ as a scheme over $B_{A_2}$. Therefore the total space $\sX_{\pA_2}\subset \Aa^{5}_{\theta_i}\times \Aa^{10}_{A_i,B_i}$ is an affine variety of codimension 8 defined by 15 equations. This variety is Cohen--Macaulay, but not Gorenstein.

\subsubsection{The cluster variety $X_{\pA_2}$} 
To recover the Grassmannian $\Gr(2,5)$ we restrict $\sX$ to the locus $\sX|_{\TT^5}\subset\sX$ over the structure torus $\TT^5\subset B$. After writing all of the elements of $N_1(Y)$ in terms of the basis $[D_i]$, the equations become
\[ \theta_{i-1}\theta_{i+1} = A_i\theta_i + A_{i-2}A_{i+2} \] 
which we see to be ideal given by the $4\times4$ maximal Pfaffians of a $5\times 5$ skewsymmetric matrix.
\[ \Pf_4\begin{pmatrix}
A_5 & \theta_1 & \theta_2 & A_3 \\ 
 & A_2 & \theta_3 & \theta_4 \\
 & & A_4 & \theta_5 \\
 & & & A_1
\end{pmatrix} \]
Taking the closure of $\sX|_{\TT^5}$ over $\TT^5\subset\Aa^5_{A_i}$ gives the cluster variety $X_{\pA_2}$. Indeed, we see that $X_{\pA_2}$ is nothing other than the affine cone over the Grassmannian $\Gr(2,5)$ in its Pl\"ucker embedding.

\subsubsection{Symmetries} 
$X_{\pA_2}$ has the action of the group $\Dih_{10}\times \TT^D$, where $\Dih_{10}$ permutes the indices $\{1,\ldots,5\}$. The characters for the $\TT^D$-action are given by $\chi_i\left(z^{C}\right)=D_i\cdot C$, as shown in Table~\ref{table!weight_table_for_A2}.
\begin{table}[h]
\caption{The character table for $\TT^D \curvearrowright X_{\pA_2}$.}
\begin{center}
$\begin{array}{|c|ccccc|ccccc|}\hline
 & \theta_1 & \theta_2 & \theta_3 & \theta_4 & \theta_5 & A_1 & A_2 & A_3 & A_4 & A_5 \\\hline
\chi_1 & 1 & 0 & 0 & 0 & 0 &  -1 & 1 & 0 & 0 & 1 \\
\chi_2 & 0 & 1 & 0 & 0 & 0 &   1 &-1 & 1 & 0 & 0 \\
\chi_3 & 0 & 0 & 1 & 0 & 0 &   0 & 1 &-1 & 1 & 0 \\
\chi_4 & 0 & 0 & 0 & 1 & 0 &   0 & 0 & 1 &-1 & 1 \\
\chi_5 & 0 & 0 & 0 & 0 & 1 &   1 & 0 & 0 & 1 &-1 \\ \hline
\end{array}$
\end{center}
\label{table!weight_table_for_A2}
\end{table}%

\subsection{Cluster varieties as key varieties}\label{section!cluster-formats}
Suppose that $X=\Spec\sA\subset \Aa^n$ is an affine cluster variety with torus action $\TT^k\times X\to X$. Define the character lattice $M=\Hom(\TT^k,\TT)\cong\ZZ^k$ and the dual lattice of one parameter subgroups $M^\vee=\Hom(\TT,\TT^k)$, together with the perfect pairing ${\left<{}\cdot{},{}\cdot{}\right>}\colon M\times M^\vee\to\ZZ$. The following objects are all endowed with an $M$-grading induced by the torus action: the coordinate ring $\sA=\bigoplus_{\chi\in M}\sA_\chi$, the ambient ring $\sO_{\Aa^n}$, the minimal free resolution $\FF$ of $\sA$ as an $\sO_{\Aa^n}$-module and the Hilbert series of $X$
\[ P_X(t_1,\ldots,t_k)=\sum_{\chi\in M}\dim(\sA_{\chi})t_1^{\chi_1}\cdots t_k^{\chi_k}.\] 
Following the definition of a Gorenstein format by Brown, Kasprzyk \& Zhu \cite{bkz}, we make the following definition:

\begin{defn}\label{definition!cluster-format}
A \emph{cluster format} is a triple $(X,\mu,\FF)$ where $X\subset \Aa^n$ is a cluster variety, $\mu$ is the character of an action $\TT \curvearrowright X$ and $\FF$ is a $\ZZ$-graded resolution of $\sA$ as an $\sO_{\Aa^n}$-module. If $X$ is the cluster variety of finite type $\mathsf{T}$ we also call this \emph{$\mathsf{T}$ format}.
\end{defn}

In this setup, a cluster format is determined by the choice of cluster variety $X$ and a one parameter subgroup $\rho \in M^\vee$. For such a $\rho$, the action of $\lambda\in\TT$ on $v_\chi \in \sA_{\chi}$ is given by $\lambda\cdot v_\chi\mapsto \lambda^{\langle\rho,\chi\rangle}v_\chi$, and extended to all of $\sA$ linearly. The \emph{degree} of $v_\chi$ is this exponent, denoted $d(v_\chi):=\langle\rho,\chi\rangle$. Thus $\rho$ induces a $\ZZ$-grading on $\sA=\bigoplus_{d\in\ZZ} \sA_d$, where $\sA_d=\bigoplus_{\{\chi\in M : \left<\chi,\rho\right>=d\}}\sA_{\chi}$. The polynomial ring $\sO_{\Aa^n}$ is $\ZZ$-graded in a similar way,
which determines the character $\mu=\langle\rho,\cdot\rangle$ of the $\TT$-action and a $\ZZ$-grading on $\FF$.

Fix a cluster format $(X,\mu,\FF)$ of codimension $c$ and consider the polynomial ring $\sO_{\Aa^m}=\CC[y_1,\dots,y_m]$ with a (positive) $\ZZ$-grading $(a_1,\dots,a_m)$. Let $\phi\colon \Aa^m\to\Aa^n$ be a homogeneous morphism of degree zero with respect to the given grading on $\sO_{\Aa^m}$ and the $\mu$-grading on $\sO_{\Aa^n}$ (i.e.\ $\phi$ is $\TT$-equivariant).

\begin{propdefn}[cf.~{\cite[Proposition 1.3]{Fun}}]
Let $\widehat Y=\phi^{-1}(X)\subset\Aa^m$ for a morphism $\phi\colon \Aa^m\to\Aa^n$ homogeneous of degree zero, as above. Then $\widehat Y$ is called a \emph{regular pullback} of $X$ if one of the following equivalent conditions hold:
\begin{enumerate}
\item $\widehat Y\subset\Aa^m$ has codimension $c$;
\item the pullback of $\FF$ by $\phi$ is a free resolution of $\sO_{\Aa^m}$-modules;
\item if $x_i$ are coordinates on $\Aa^n$, then $x_i-\phi^*(x_i)$ form a regular sequence on $\Aa^m\times\Aa^n$ for $i=1,\dots,n$.
\end{enumerate}
\end{propdefn}
The point of the definition is that all of the equations, the syzygies, the 
Hilbert series etc., of $\widehat Y$ come from the cluster format $(X,\mu,\FF)$ 
together with the morphism $\phi$. Since $\phi$ is $\TT$-equivariant, we may 
define the weighted projective variety associated to $(X,\mu,\FF)$ and $\phi$ by 
taking the GIT quotient
\[ Y = \left( \widehat{Y} \: {/\!\!/}_{\!\mu} \: \TT \right) \subset\PP(a_1,\dots,a_m).\]
See Examples \ref{eg!first-C2} and \ref{eg!first-G2} for details.

\begin{rmk} \leavevmode
\begin{enumerate}
\item The character $\mu$ is allowed to have non-positive weights, since if $x_i$ is a coordinate with $d(x_i)<0$ then $\phi^*(x_i)=0$. 
\item By considering larger torus actions $\TT^d$ for $d\leq k$, we could also use $X$ as a key variety for the Cox ring of some VGIT quotient, e.g.~to construct 3-fold flips as with Brown \& Reid's diptych varieties \cite{dip2}, divisorial extractions as in \cite{phd}, or Sarkisov links in the style of Brown \& Zucconi \cite{bz}.
\end{enumerate}
In this paper we consider the first generalisation. Our convention is always to assume that $\phi$ is a generic morphism, and therefore that $\phi^*(v)\neq0$ is a non-zero constant if $d(v)=0$.
\end{rmk}


We end this section with a useful lemma.

\begin{lem}[Singularity avoidance lemma]\label{lemma!singularity-avoidance}
Let $\widehat{Y}=\phi^{-1}(X)$ be a regular pullback where $\phi\colon \Aa^m \to \Aa^n$ is a morphism of graded degree zero.
\begin{enumerate}
\item $\phi^{-1} \big(\sing(X) \big) \subseteq \sing\big(\widehat Y\big)$
\item Let $\Pi=V(f_1,\ldots,f_c)\subset \Aa^n$ be a homogeneous complete intersection of codimension $c\leq m$. Then either: 
\begin{enumerate}
\item $\phi^{-1}(\Pi) = \emptyset$, which happens if and only if $d(f_i)=0$ and $\phi^*(f_i)\neq0$ for some $i$, or
\item $\dim\phi^{-1}(\Pi)\geq m-c$. 
\end{enumerate}
\item If $\widehat Y$ is the affine cone over a quasismooth weighted projective variety $Y$, then $\phi^{-1} \big( \sing(X) \big)$ is at worst the cone point $P\in \widehat Y$. Moreover if $\Pi\subseteq \sing(X)$ is a homogeneous complete intersection in $\Aa^n$ of codimension $< m$, then $d(f)=0$ and $\phi^*(f)\neq0$ for some generator $f\in I(\Pi)$.
\end{enumerate}
\end{lem}

\begin{proof}
Suppose $X$ is defined by equations $g_1,\ldots,g_d$ in variables $x_1,\ldots,x_n$ and $\widehat Y$ is defined by equations $h_1,\ldots,h_d$ in variables $y_1,\ldots,y_m$, where $h_i(y_1,\ldots,y_n) = g_i(\phi^*(x_1), \ldots, \phi^*(x_n))$ for all~$i$. Now, by the chain rule for differentation, we have
\[ \Jac(\widehat Y) = \left(\frac{\partial h_i}{\partial y_j}\right) = \phi^*\left(\frac{\partial g_i}{\partial x_k}\right)\cdot\left(\frac{\partial\phi^*(x_k)}{\partial y_j}\right) = \phi^*(\Jac(X))\cdot \Jac(\phi) \]
and when the rank of $\Jac(X)$ is less than $c$ then the rank of $\Jac(\widehat Y)$ must be less than $c$. This proves statement (1).

Statement (2) follows from $\phi^{-1}(\Pi) = V\left(\phi^*(f_1),\ldots,\phi^*(f_c)\right) \subset \Aa^n$, which is an intersection of $c$ homogeneous polynomials in $\Aa^m$. These define a locus of dimension $\geq m-c$, unless one $\phi^*(f_i)$ is identically nonzero. This can only happen if $d(f_i)=0$ and $\phi^*(f_i)\neq0$. 

Statement (3) follows directly from (1) and (2).
\end{proof}

\begin{rmk}
In our situation, $\phi\colon\Aa^m\to\Aa^n$ is usually a generic immersion. One might ask whether $\phi^{-1}\big(\sing(X)\big)$ being empty implies that $\sing\big(\widehat Y\big)$ is empty. This is not true; the rank of $\Jac(\widehat Y)$ may drop if the image of $\Jac(\phi)$ intersects too much of the kernel of $\phi^*(\Jac(X))$.
See \S\ref{sec!failure} for examples.
\end{rmk}

\begin{rmk}
Our codimension 4 cluster formats determine certain loci inside $\text{SpH}_8$, the Spin-Hom variety introduced by Reid in \cite{Reid-codim-4}. The main theorem of \cite{Reid-codim-4} puts codimension 4 Gorenstein ideals $I$ into correspondence with regular pullbacks by suitable morphisms $\varphi\colon\Aa^n\to\text{SpH}_k$, thus $\text{SpH}_k$ acts as a key variety for the $(k+1)\times 2k$ first syzygy matrix of $I$. We specify a grading on $\Mor(\Aa^8,\text{SpH}_8)$ and only consider those morphisms landing in the cluster locus. We classify the components of this space which correspond to quasismooth varieties. This is a tractable case of a question raised in \cite[\S4.8]{Reid-codim-4}.
\end{rmk}


\section{$\pC_2$ cluster format} \label{section!C2-cluster-format}

Recall that the cluster variety $X_{\pC_2}\subset \Aa^{13}$ is the affine Gorenstein $9$-fold in codimension $4$ described in \S\ref{subsubsection!C2-format}. We now describe how to derive the equations \eqref{eqns!C2-format} defining $X_{\pC_2}$ from the mirror family $\sX_{\pC_2}$. Throughout the whole of this section we consider subscripts $(i,j,k)$ in all formulae to vary over all of the $\Dih_6$-permutations of $(1,2,3)$.

\subsection{The equations for $\pC_2$ format} \label{section!C2-format}

Recall that the mirror family $\sX_{\pC_2}$ is defined over a toric basic variety  $B_{\pC_2}=\Spec\left(\CC[\NE(Y_{\pC_2})]\right)$.

\subsubsection{The toric base $B_{\pC_2}$}
In this case $B_{\pC_2}$ is a singular affine toric variety with 18 equations of the form $z^X=z^Y$, where $X=Y$ is a linear relation in $N_1(Y_{\pC_2})$ for some classes $X,Y\in\NE(Y_{\pC_2})$. We only write down six of these 18 equations, which will be relevant to the following calculation:
\begin{equation} \label{eqns!C2-toric-base}
\begin{aligned} \relax
[D_i]+2[E_i]+[F] &= [D_j]+2[D_{jk}]+[D_k], \\ \relax
[D_{ij}]+[E_{ij}] &= [D_{jk}]+[D_k]+[D_{ki}]. 
\end{aligned}\end{equation}

\subsubsection{The mirror family $\sX_{\pC_2}$}
In this case, the mirror family $\sX_{\pC_2}/B_{\pC_2}$ is defined by nine relative equations. These nine equations are determined by the six tag equations:
\begin{align*}
\theta_i\theta_j &= z^{D_{ij}}\left(\theta_{ij} + z^{E_{ij}}\right) \\
\theta_{ij}\theta_{jk} &= z^{D_j}\left(\theta_j + z^{E_j}\right)\left(\theta_j + z^{E_j+F}\right) 
\end{align*}
where the monomials appearing in the righthand side of these equations come from counting certain classes of rational curves on $Y_{\pC_2}$. In general, the expectation that the coefficients appearing in the mirror algebra can be interpreted in terms of the enumerative geometry of $Y$ is described in \cite{gs}.  

In the case above, the first tag equation is of the form $\theta_i\theta_j=\sum_{m=1}^2z^{[\Sigma_m]}\theta_{ij}^{-D_{ij}\cdot\Sigma_m}$, where $[\Sigma_1]=[D_{ij}]$ and $[\Sigma_2]=[D_{ij}]+[E_{ij}]$ are the two classes of an effective rational curve $\Sigma\subset Y_{\pC_2}$ such that $\Sigma \cdot D_i=\Sigma \cdot D_j =1$, and $\Sigma\cdot D'=0$ for all other irreducible components in the boundary $D'\subset D$. Similarly the second tag equation is $\theta_{ij}\theta_{jk}=\sum_{m=1}^4z^{[\Sigma_m]}\theta_j^{-D_j\cdot\Sigma_m}$, where $[\Sigma_1]=[D_j]$, $[\Sigma_2]=[D_j]+[E_j]$, $[\Sigma_3]=[D_j]+[E_j]+[F]$ and $[\Sigma_4]=[D_j]+2[E_j]+[F]$ are the four classes of an effective rational curve $\Sigma\subset Y_{\pC_2}$ such that $\Sigma \cdot D_{ij}=\Sigma \cdot D_{jk} =1$, and $\Sigma\cdot D'=0$ for all other irreducible components in the boundary $D'\subset D$.
The remaining equations, which are of the form $\theta_i\theta_{jk}=\cdots$, can 
either be obtained by a similar calculation (i.e.\ finding the relevant classes 
of rational curves passing between $D_i$ and $D_{jk}$), or by simply calculating 
the relation which is implied birationally from the tag equations.\footnote{In 
much the same way that the tag equations of a toric variety determine all of 
the other equations.}

\begin{rmk} \label{rmk!no-proper-proof}
Since we are primarily concerned with the existence of $X_{\pC_2}$ we do 
not take the time to give a rigorous proof of this description. To do that one 
would either have to calculate the relevant Gromov--Witten invariants for $Y$ or 
(similarly to \cite[Example 3.7]{ghk}) show that there is a consistent 
scattering diagram with six rays, corresponding to the six cluster variables, 
with the attached functions $z^{D_{ij}}\left(1 + 
z^{E_{ij}}\theta_{ij}^{-1}\right)$ and $z^{D_j}\left(1 + 
z^{E_j}\theta_j^{-1}\right)\left(1 + z^{E_j+F}\theta_j^{-1}\right)$.
\end{rmk}

\subsubsection{The cluster variety $X_{\pC_2}$} 
We write $N_1(Y_{\pC_2})=\bm{D} \oplus \ZZ[\delta]$, according to the $\QQ$-basis $[D_1],\ldots,[D_{31}],\delta$, where 
\[ \delta=\tfrac12\left([D_1]+[D_2]+[D_3]-[F]\right) = \pi_{\pC_2}^*H - D_{12} - D_{23} - D_{31} \]
where $\pi_{\pC_2}$ is as in \S\ref{sect!toric-models} and $H$ is the hyperplane class on $\PP^2$.
The reason for this choice of basis is that, by the equations for $B_{\pC_2}$ \eqref{eqns!C2-toric-base}, we have:
\begin{align*}
[E_i] & = [D_{jk}] - [D_i] + \delta, \\
[E_{ij}] &= [D_{jk}] + [D_k] + [D_{ki}] - [D_{ij}],
\end{align*}
which allows us to eliminate the coefficients $z^{E_i}$, $z^{E_{ij}}$ in a 
$\Dih_6$-invariant way. After doing this, and setting $A_i:=z^{D_i}$, 
$A_{ij}:=z^{D_{ij}}$ and $\lambda:=z^{\delta}(1+z^F)$ to simplify the notation, 
we arrive at our desired equations \eqref{eqns!C2-format}, albeit defined over 
the torus $\TT^7_{A_1,\ldots,A_{31},\lambda}$. Since all the exponents that 
appear in the equations are positive and integral, the equations defining 
${\sX_{\pC_2}}|_{\TT^7}$ immediately extend to obtain the cluster variety 
$X_{\pC_2}/\Aa^7$.

\subsubsection{Symmetries} 
The cluster variety $X_{\pC_2}$ has the action of $\Dih_6\times\TT^6$, where $\Dih_6$ permutes the indices $\{1,2,3\}$. The torus action $\TT^6=\TT^D \curvearrowright X_{\pC_2}$ is determined by the six characters $\chi_i$, $\chi_{ij}$, as defined in \S\ref{sect!curlyX}. Since $\delta\cdot[D_i]=-1$ for all $i$ and $\delta\cdot[D_{ij}]=1$ for all $i,j$, the character table for $\TT^D \curvearrowright X_{\pC_2}$ is given by Table~\ref{table!weight_table_for_C2}.
\begin{table}[h]
\caption{The character table for $\TT^D\curvearrowright X_{\pC_2}$.}
\begin{center}
$\begin{array}{|c|cccccc|ccccccc|}\hline
 & \theta_1 & \theta_{12} & \theta_2 & \theta_{23} & \theta_3 & \theta_{31} & A_1 & A_{12} & A_2 & A_{23} & A_3 & A_{31} & \lambda \\\hline
\chi_1    &  1 & 0 & 0 & 0 & 0 & 0 &-2 & 1 & 0 & 0 & 0 & 1 &-1 \\
\chi_{12} &  0 & 1 & 0 & 0 & 0 & 0 &1 &-1 & 1 & 0 & 0 & 0 & 1 \\
\chi_2    &  0 & 0 & 1 & 0 & 0 & 0 &0 & 1 &-2 & 1 & 0 & 0 &-1 \\
\chi_{23} &  0 & 0 & 0 & 1 & 0 & 0 &0 & 0 & 1 &-1 & 1 & 0 & 1 \\
\chi_3    &  0 & 0 & 0 & 0 & 1 & 0 &0 & 0 & 0 & 1 &-2 & 1 &-1 \\
\chi_{31} &  0 & 0 & 0 & 0 & 0 & 1 &1 & 0 & 0 & 0 & 1 &-1 & 1 \\\hline
\end{array}$
\end{center}
\label{table!weight_table_for_C2}
\end{table}%

\subsection{Alternative presentations for $\pC_2$ format}

The nine equations~\eqref{eqns!C2-format} can be presented in a number of different ways.

\subsubsection{Crazy Pfaffian format}\label{section!crazy}
The equations can be written in a $6\times 6$ \emph{crazy Pfaffian} format:
\[ \renewcommand*{\arraystretch}{1.5}
\Pf_4\begin{pmatrix}
A_3A_{31} & \theta_1 & \theta_{12} & A_2\theta_2 + \lambda A_{31} & A_2A_{23}A_3 + \lambda \theta_1 \\
 & A_{12} & \theta_2 & \theta_{23} & A_3\theta_3 + \lambda A_{12} \\
 & & A_{23} & \theta_3 & \theta_{31} \\
 & & & A_{31}A_1 &  A_1\theta_1 \\
 & & & & A_1A_{12}A_2
\end{pmatrix} \]
where the variables $A_1,A_2,A_3$ are \emph{floating factors}. In other words, after expanding these Pfaffians some of the relations are found to be divisible by $A_1$, $A_2$ or $A_3$. In crazy Pfaffians format we allow ourselves to divide by these floating factors wherever possible. In particular if we set $A_1=A_2=A_3=1$ and $\lambda=0$ then we recover the codimension 4 \emph{extrasymmetric format} which first appeared in Dicks' thesis \cite{dicks}, and now in many other places.

\subsubsection{Triple unprojection structure}\label{subsection!triple-unprojection}
Eliminating $\theta_{12},\theta_{23},\theta_{31}$ from $\sA_{\pC_2}$ gives a Gorenstein projection $X_{\pC_2}\dashrightarrow Z$ where $Z$ is the hypersurface:
\[ \theta_1\theta_2\theta_3 = A_{31}A_1A_{12}\theta_1 + A_{12}A_2A_{23}\theta_2 + A_{23}A_3A_{31}\theta_3 + \lambda A_{12}A_{23}A_{31} \]
This variety $Z$ is a family of affine cubic surfaces over $\Aa^7_{A_i,A_{ij},\lambda}$ whose general member has three lines at infinity meeting at three $\tfrac12(1,1)$~singularities, obtained by contracting the three $(-2)$-curves in the boundary divisor of $Y_{\pC_2}$. The variable $\theta_{ij}$ can be recovered from $Z$ as a serial Gorenstein type I unprojection of the divisor $\Pi_{ij}=V(A_{ij},\theta_k)$. This gives rise to the following description as an interlaced $4\times 4$-Pfaffians format for the three matrices:
\begin{equation}\label{eq!C2-tom}
\Pf_4 \begin{pmatrix}
A_kA_{ki} & \theta_i & \theta_{ij} & A_j\theta_j + \lambda A_{ki} \\
 & A_{ij} & \theta_j & \theta_{jk} \\
 & & A_{jk} & \theta_k \\
 & & & A_{ki}A_i
\end{pmatrix} 
\end{equation}
where two Pfaffian equations in each matrix are repeated in one of the other two 
matrices. From any one of these three matrices, $X_{\pC_2}$ is given by 
unprojecting the Tom$_{3}$ ideal $(A_{ki},\theta_{ij},\theta_{j},\theta_{jk})$ 
with unprojection variable $\theta_{ki}$.

\subsubsection{Papadakis \& Neves' $\binom{n}{2}$ Pfaffians format}
Papadakis \& Neves \cite{pn} define the \emph{$\binom n2$ Pfaffians format} as a series of parallel type I unprojections from a certain codimension 1 ring. When $n=3$ (and in different notation from \cite{pn}) it is given by the parallel unprojection of the three ideals $(u_1,v_1)$, $(u_2,v_2)$, $(u_3,v_3)$ contained in the hypersurface:
\[ Cu_1u_2u_3 - D_1v_1u_2u_3 - D_2u_1v_2u_3 - D_3u_1u_2v_3 + E_1u_1v_2v_3 + E_2v_1u_2v_3 + E_3v_1v_2u_3 - Fv_1v_2v_3 = 0 \]
The result is a Gorenstein ring in codimension 4 with $9\times16$ equations and syzygies. For $(i,j,k)$ any $\Dih_6$-permutation of $(1,2,3)$, the nine equations are:
\begin{align*}
\tag{$\times3$} w_iu_i &= D_iu_ju_k - E_ju_jv_k - E_ku_kv_j + Fv_jv_k \\
\tag{$\times3$} w_iv_i &= Cu_ju_k - D_ju_kv_j - D_ku_jv_k + E_iv_jv_k \\
\tag{$\times3$} w_jw_k &= (D_jD_k - CE_i)u_i^2 + (CF + D_iE_i - D_jE_j - D_kE_k)u_iv_i + (E_jE_k - D_iF)v_i^2 
\end{align*}
and (as can be seen from the hypersurface model $Z$ of \S\ref{subsection!triple-unprojection}) if we set: 
\[ (u_i,\; v_i,\; w_i;\; C,\; D_i,\; E_i,\; F)\mapsto(\theta_i,\; A_{jk},\; \theta_{jk};\; 1,\; 0,\; -A_i,\; \lambda) \]
then we recover $X_{\pC_2}$. The $\binom 32$ Pfaffians ring has symmetry group\footnote{$\BDih_6$ is the binary dihedral group---i.e.\ a central extension of $D_6$ of order 2. In this case the extra involution switches $u_i\leftrightarrow v_i$, $D_i\leftrightarrow E_i$ and $C\leftrightarrow F$.} $\BDih_6\times \TT^7$ which is slightly larger than $\Dih_6\times \TT^6$, the symmetry group of $X_{\pC_2}$.

\begin{rmk}
The reason that we stick to the cluster algebra format and do not consider this more general format is mainly due to computational advantage. Even though this ring is not that much bigger than $\sA_{\pC_2}$ (and has greater symmetry) in almost all computations the computer has a much harder time working with it. For example, the decomposition of $\binom 32$ Pfaffian format into affine charts is more complicated than for $X_{\pC_2}$, which is worked out next.

\emph{Question}: Can we also obtain the $\binom{3}{2}$ Pfaffians variety from $\sX_{\pC_2}$? It seems a little suspicious that the rank of the torus action is now bigger, and that part of the symmetry switches cluster variables $\theta_{ij}$ with coefficients $A_{ij}$.
\end{rmk}

\subsection{Affine pieces and the deep locus}\label{subsection!C2-affine-pieces}
We explain in more detail the partial covering of the $\pC_2$ cluster variety from Lemma \ref{lem!partial-covering}. In the locus where the cluster variable $\theta_{12}$ does not vanish, the equations defining $X_{\pC_2}\cap (\theta_{12}\neq0)$ reduce to $\CI^{(4)}$:
\begin{align*}
\theta_{31}\theta_{12} &= A_1\theta_1^2 + \lambda A_{23}\theta_1 + A_2A_{23}^2A_3  \\
\theta_{12}\theta_{23} &= A_2\theta_2^2 + \lambda A_{31}\theta_2 + A_3A_{31}^2A_1 \\
\theta_1\theta_2 &= A_{12}\theta_{12} + A_{23}A_3A_{31}  \\
\theta_3\theta_{12} &= A_{31}A_1\theta_1 + \lambda A_{23}A_{31} + A_2A_{23}\theta_2
\end{align*}
Similarly if any of the other cluster variables $\theta_i$, $\theta_{ij}$ are nonvanishing, the equations also reduce to $\CI^{(4)}$. Therefore $X_{\pC_2}$ is partly covered by six affine $\CI^{(4)}$ charts and the remaining `deep locus' $X_0=X\cap V(\theta_1,\theta_{12},\theta_2,\theta_{23},\theta_3,\theta_{31})$ decomposes into 11 pieces. Up to the $\Dih_6$ symmetry, these are:
\begin{align*}
\tag{$\times 1$} \Aa^4_{A_1,A_2,A_3,\lambda} &= V(\theta_1,\ldots,\theta_{31},A_{12},A_{23},A_{31}) \\
\tag{$\times 3$} \Aa^4_{A_{23},A_{31},A_1,A_2} &= V(\theta_1,\ldots,\theta_{31},A_3,A_{12},\lambda) \\
\tag{$\times 6$} \Aa^4_{A_{31},A_2,A_3,\lambda} &= V(\theta_1,\ldots,\theta_{31},A_1,A_{12},A_{23}) \\
\tag{$\times 1$} \Aa^3_{A_{12},A_{23},A_{31}} &= V(\theta_1,\ldots,\theta_{31},A_1,A_2,A_3,\lambda) 
\end{align*}

\subsection{Regular pullbacks from $\pC_2$ format}
\label{eg!first-C2} 

Let $(X_{\pC_2},\mu,\FF)$ be a $\pC_2$ format determined by the one parameter subgroup 
$$\rho=(\rho_1,\dots,\rho_{31})\colon\CC^*\to\TT^D.$$
The action on $X_{\pC_2}$ is readily computed from 
Table~\ref{table!weight_table_for_C2}; we just multiply the matrix of torus 
weights on the left by $\rho$ to obtain: $d(\theta_i)=\rho_i$, 
$d(\theta_{ij})=\rho_{ij}$, $d(A_i)=\rho_{ki}-2\rho_i+\rho_{ij}$, 
$d(A_{ij})=\rho_i-\rho_{ij}+\rho_j$ and $d(\lambda)=\rho_{12} + \rho_{23} + 
\rho_{31} - \rho_1 - \rho_2 - \rho_3$.

We use the following shorthand to write down a regular pullback from $\pC_2$ format:
\[ \pC_2\left(\begin{array}{ccc|ccc|}
\phi^*(\theta_{12}) & \phi^*(\theta_{23}) & \phi^*(\theta_{31}) & \phi^*(\theta_1) & \phi^*(\theta_2) & \phi^*(\theta_3) \\
\phi^*(A_{12}) & \phi^*(A_{23}) & \phi^*(A_{31}) & \phi^*(A_1) &  \phi^*(A_2) & \phi^*(A_3)
\end{array} \:\: \phi^*(\lambda) \right) \]
and the same array with integer entries if we wish to denote a generic pullback with given degrees. 

The $M$-graded Hilbert series of $X_{\pC_2}$ can be computed using Macaulay2 (or even by hand), and we can easily translate this into the $\ZZ$-graded Hilbert series of $(X_{\pC_2},\mu,\FF)$:
\[ P_{(X_{\pC_2},\mu,\FF)}(t) = P_{X_{\pC_2}}(t^{\rho_1},t^{\rho_{12}},t^{\rho_2},t^{\rho_{23}},t^{\rho_3},t^{\rho_{31}}) \]
The resolution $\FF$ is Gorenstein codimension four with 9 relations and 16 syzygies, and the Hilbert numerator is of the form:
\[ 1 - \sum(t^{\rho_i+\rho_j}+t^{\rho_{ij}+\rho_{jk}}+t^{\rho_i+\rho_{jk}}) + \cdots + t^{\alpha} \]
where the adjunction number is $\alpha=\rho_1+\rho_{12}+\rho_2+\rho_{23}+\rho_3+\rho_{31}$. 

\subsection{Singular locus}


We want to construct quasismooth 3-dimensional varieties via regular pullback from a key variety $X$ that turns out to be rather singular. According to Lemma~\ref{lemma!singularity-avoidance}, we have to control the dimension of the pullback of $\sing(X)$, so we first compute the singular locus of $X_{\pC_2}$ and of some distinguished subvarieties of $X_{\pC_2}$.

\begin{lem} \label{lemma!C2_singular_locus}
The reduced singular locus of $X_{\pC_2}$ is contained in the deep locus \[ \sing(X_{\pC_2}) \subset X_0 = X_{\pC_2}\cap V(\theta_1,\theta_{12},\theta_2,\theta_{23},\theta_3,\theta_{31}) \] 
and decomposes into four irreducible linear subvarieties, given by:
\[ \Aa^4_{A_1,A_2,A_3,\lambda} = X_0\cap V(A_{12},A_{23},A_{31}) \quad \text{and} \quad \Aa^2_{A_i,A_{jk}} = X_0\cap V(A_{ij},A_j,A_k,A_{ki},\lambda). \]
In particular all components of the singular locus have codimension $\geq5$ in $X_{\pC_2}$.

Moreover, the singular locus of the hyperplane section $X^z := X_{\pC_2}\cap V(z)$ contains the following bad components of codimension $\geq3$.
\begin{enumerate}
\item $\sing\left(X^{\theta_i}\right)$ is contained in the locus $X^{\theta_i}_0:=X^{\theta_i}\cap V(\theta_{ij},\theta_j,\theta_k,\theta_{ki})$ and contains the following component which has codimension 3 in $X^{\theta_i}$:
\[ \Aa^5_{\theta_{jk},A_i,A_j,A_k,\lambda} = X^{\theta_i}_0 \cap V(A_{ij},A_{jk},A_{ki}). \]
\item $\sing\left(X^{\theta_{ij}}\right)$ contains the following component which has codimension 1 in $X^{\theta_{ij}}$:
\[ X^{\theta_{ij}}\cap V(\theta_i,\theta_j,A_{jk},A_{ki}). \]
\item $\sing\left(X^{A_i}\right)$ is contained in the locus $X^{A_i}_0:=X^{A_i}\cap V(\theta_{ij},\theta_j,\theta_{jk},\theta_k,\theta_{ki})$ and contains the following component which has codimension 3 in $X^{A_i}$:
\[ \Aa^5_{\theta_i,A_{ij},A_j,A_k,A_{ki}} = X^{A_i}_0 \cap V(A_i,A_{jk},\lambda). \]
\item $\sing\left(X^{A_{ij}}\right)$ is contained in the locus $X^{A_{ij}}_0:=X^{A_{ij}}\cap V(\theta_j,\theta_{jk},\theta_k,\theta_{ki},\theta_i)$ and contains the following component which has codimension 3 in $X^{A_{ij}}$:
\[ \Aa^5_{\theta_{ij},A_i,A_j,A_k,\lambda} = X^{A_{ij}}_0 \cap V(A_{ij},A_{jk},A_{ki}). \]
\end{enumerate}
\end{lem}

\begin{proof}
The statements about singular loci can easily be checked by using Macaulay2 or Magma (cf.\ \cite[Theorem 1.1]{God3}). Note that if $\theta_{ij}=0$ then $T = \frac{A_i\theta_i}{A_{jk}} = \frac{A_j\theta_j}{A_{ki}}$ are solutions to the equation $T^2 + \lambda T + A_1A_2A_3 = 0$ in the ring $\sO_{X^{\theta_{ij}}}$, so $X^{\theta_{ij}}$ is not normal. Since $X^{\theta_{ij}}$ is Gorenstein, and hence $S_2$, it must be singular in codimension 1.
\end{proof}

\subsection{Quasismoothness conditions}

Let $\widehat Y \subset \Aa^8$ be the 4-dimensional affine cone over a quasismooth weighted projective 3-fold $Y\subset w\PP^7$ given as a regular pullback of the $\pC_2$ cluster variety. The following Lemma lists the conditions imposed on the format by avoiding large components in the singular locus of $X_{\pC_2}$.

\begin{prop} \label{proposition!C2_subformats}
Suppose that $\widehat Y := \phi^{-1}(X_{\pC_2})$ is a regular pullback and is not a complete intersection. Then, for all $i,j$, we must have $d(\theta_i)>0$, $d(\theta_{ij})>0$ and one of the following conditions must hold:
\begin{enumerate} \setlength\itemsep{0.5em}
\item If $d(A_{ij})>0$ for all $i,j$ and $d(\lambda)>0$, then we say $\widehat{Y}$ is in {\bf `$\pC_2$ format'}. In this case, $d(A_i)\geq 0$ for all~$i$.
\item If $d(A_{ij})>0$ for all $i,j$ and $d(\lambda)=0$, then we say $\widehat Y$ is in {\bf `$\PP^2\times \PP^2$ format'}. In this case, either
\begin{enumerate}
\item $d(A_1)=d(A_2)=d(A_3)=0$, or
\item $d(A_i)<0$ for some $i$.
\end{enumerate}
\item If $d(A_{ij})=0$ then we say $\widehat{Y}$ is in {\bf `$\pA_2+\CI^{(1)}$ 
format'}. In this case, all of $A_i$, $A_{jk}$, $A_{ki}$, $A_j$, $\lambda$ have 
degree~$>0$.
\end{enumerate}
\end{prop}

\begin{proof}
If $d(\theta_i)=0$ for some $i$, or if $d(\theta_{ij})=0$ for some $i,j$, then we can eliminate most of the equations to be left with a $\CI^{(4)}$. Therefore we may assume that $d(\theta_i)\neq0$ and $d(\theta_{ij})\neq0$ for all~$i,j$. 

We now prove that the stated conditions on the degrees of the variables hold through the following series of claims. We repeatedly use the following argument: if a variable $z$ has degree $d(z)<0$ then $\phi^*(z)=0$, and hence $\phi$ must factor as a regular pullback from $X^z$. Then, by Lemma~\ref{lemma!singularity-avoidance}(3), some of the other variables must be non-vanishing and constant in order to avoid pulling back the bad components in $\sing\left(X^z\right)$ of codimension $\geq3$, listed in Lemma~\ref{lemma!C2_singular_locus}.

{\bf Claim 1}: Any one of $d(\theta_i)<0$, $d(\theta_{ij})<0$ or $d(A_{ij})<0$ cannot happen. 

If $d(\theta_i)<0$ then $\phi$ factors through $X^{\theta_i}$ and, in order to avoid pulling back the bad component of Lemma~\ref{lemma!C2_singular_locus}(1), we must have either $d(A_{ij})=0$, $d(A_{jk})=0$ or $d(A_{ki})=0$. This puts us in case (3) below, but with a zero entry appearing in the Pfaffian matrix. Hence $\widehat{Y}$ will fail to be quasismooth, by \cite[Proposition 2.7]{bkz}. Similarly for the cases $d(\theta_{ij})<0$ and $d(A_{ij})<0$.

{\bf Claim 2}: If $d(A_i)<0$ for some $i$, then we are either in case (2.b) or case (3).

To avoid pulling back the bad component of $\sing\left(X^{A_i}\right)$ we need either $d(\lambda)=0$ which puts us in case (2.b), or $d(A_{jk})=0$ which puts us in case $(3)$.

{\bf Claim 3}: If $d(\lambda)=0$ then we are in in case (2), and $d(\lambda)<0$ cannot happen.

From Table~\ref{table!weight_table_for_C2} we have the relation $2d(\lambda) = d(A_1) + d(A_2) + d(A_3)$. Therefore $d(\lambda)=0$ implies either that $d(A_1)=d(A_2)=d(A_3)=0$ or that $d(A_i)<0$ for some $i$, which are the two conditions of case (2). If $d(\lambda)<0$ then $d(A_i)<0$ for some $i$ and by Claim 2 we end up in case (3), but with a zero entry in the Pfaffian matrix. Hence $\widehat{Y}$ will not be quasismooth, as in the conclusion of Claim 1. 

This completes the analysis of the allowed degrees in cases (1)--(3). We now show that case (2) is equivalent to $\PP^2\times\PP^2$ format and case (3) is equivalent to $\pA_2+\CI^{(1)}$ format.

\paragraph{Case (2) is $\PP^2\times\PP^2$ format.} 
In case (2.a), $\phi$ factors through the regular pullback of $X_{\pC_2}$ by the morphism $\phi_1\colon \Aa^9\to \Aa^{13}$, given by $\phi_1^*(\lambda) = \phi_1^*(A_1) = \phi_1^*(A_2) = \phi_1^*(A_3) = 1$ and $\phi_1^*(z)=z$ for all other variables. If we make the change of variables $X_i=(1+\omega)(\theta_i-\omega A_{jk})$ and $Y_i=(1+\omega^2)(\theta_i-\omega^2A_{jk})$ for $\omega$ a primitive third root of unity, the equations defining $\phi_1^{-1}(X_{\pC_2})$ can be written as
\[ \bigwedge^2\begin{pmatrix}
-\theta_{12} & X_2 & Y_1 \\
Y_2 & -\theta_{23} & X_3 \\
X_1 & Y_3 & -\theta_{31} 
\end{pmatrix} = 0, \]
so that $\phi_1^{-1}(X_{\pC_2})$ is a regular pullback from $\PP^2\times\PP^2$ format. 

In case (2.b), pulling back $X_{\pC_2}$ by the morphism $\phi_2\colon \Aa^{11}\to \Aa^{13}$ given by $\phi_2^*(\lambda) = 1$, $\phi_2^*(A_i) = 0$ and $\phi_2^*(z)=z$ for all other variables, gives:
\[ \bigwedge^2\begin{pmatrix}
\theta_{ij} & A_j\theta_j + A_{ki} & A_jA_{jk}A_k + \theta_i \\
\theta_j & \theta_{jk} & A_k\theta_k + A_{ij} \\
A_{jk} & \theta_k & \theta_{ki} 
\end{pmatrix} = 0, \]
so that $\phi_2^{-1}(X_{\pC_2})$ can also be rewritten as a regular pullback from $\PP^2\times\PP^2$ format. 

\paragraph{Case (3) is $\pA_{2}+\CI^{(1)}$ format.}
In case (3), pulling back $X_{\pC_2}$ by the morphism $\phi_3\colon \Aa^{12}\to \Aa^{13}$ given by $\phi_3^*(A_{ij}) = 1$ and $\phi_3^*(z)=z$ for all other variables, gives:
\[ \Pf_4 \begin{pmatrix}
A_i & \theta_j & \theta_{jk} & A_k\theta_k + \lambda \\
 & A_{jk} & \theta_k & \theta_{ki} \\
 & & A_{ki} & \theta_i \\
 & & & A_j
\end{pmatrix} = 0 \quad \text{and} \quad \theta_{ij} = \theta_i\theta_j - A_{jk}A_kA_{ki}. \]
Therefore, $\phi_3^{-1}(X_{\pC_2})$ can be rewritten as a regular pullback from a hypersurface inside $\Gr(2,5)$ format. Note that all entries in the Pfaffian matrix must have degree $\geq0$, else $\widehat{Y}$ is too singular to be the affine cone over a quasismooth 3-fold $Y$, and if any entry has degree $0$ then $\widehat{Y}$ is a~$\CI^{(4)}$. 
\end{proof}

As a consequence of the Proposition we may easily discard cases with $d(\lambda)<0$, and if $d(\lambda)=0$ we could search with a simpler algorithm for $\PP^2\times\PP^2$ format (or just appeal to Brown, Kasprzyk \& Qureshi's work on Fano 3-folds in $\PP^2\times\PP^2$ format \cite{bkq}).

\begin{eg}[Hypersurface inside Pfaffians] 
The reason that we call case (3) of the Proposition `$\pA_2+\CI^{(1)}$ format' (and not simply `$\pA_2$ format') is that if $\phi_3^*(\theta_{ij})$ cannot be used to eliminate a variable then the variety we construct by regular pullback will be a genuine hypersurface inside $\Gr(2,5)$ format. This happens for Fano 3-fold \#29374, classically constructed as $Y=Q_2\cap\Gr(2,5)\cap\PP^7$, where $Q_2$ is a quadric hypersurface. We construct $Y$ in format $\pC_2\left(\begin{smallmatrix}2&1&1\\0&1&1\end{smallmatrix}\bigr\vert\begin{smallmatrix}1&1&1\\1&1&0\end{smallmatrix}\bigr\vert\begin{smallmatrix}1\end{smallmatrix}\right)$. Here, $d(A_{12})=0$ but, since there are no variables of degree $d(\theta_{12})=2$ to eliminate, the equation involving $\theta_{12}$ defines a quadric hypersurface. A similar phenomenon occurs for $\pG_2$ format. Indeed, \#29374 is also constructed as $\pG_2\left(\begin{smallmatrix}2&2&2&1\\0&0&0&1\end{smallmatrix}\bigr\vert\begin{smallmatrix}1&1&1&1\\0&1&1&0\end{smallmatrix}\bigr\vert\begin{smallmatrix}1\\1\end{smallmatrix}\right)$
(see Example \ref{eg!first-G2} for notation). This format commonly occurs in constructions of general type 3-folds \cite{bkz}.
\end{eg}


\section{$\pG_2$ cluster format} \label{section!G2-cluster-format}

The cluster variety $X_{\pG_2}\subset \Aa^{18}$ is the affine Gorenstein $12$-fold in codimension $6$ described in \S\ref{subsubsection!G2-format}. We now describe how to derive the equations \eqref{eqns!G2-format} defining $X_{\pG_2}$ from the mirror family $\sX_{\pG_2}$. Throughout this section we consider subscripts $(i,j,k,l)$ in all formulae to vary over all $\Dih_8$-permutations of $(1,2,3,4)$. We give a parallel treatment to the previous section
on $X_{\pC_2}$, but the situation for $X_{\pG_2}$ is more involved. Indeed, we realise $X_{\pC_2}$ as a special case of $X_{\pG_2}$.

\subsection{The equations for $\pG_2$ format} \label{section!G2-format}

Recall that the mirror family $\sX_{\pG_2}$ is defined over a toric basic variety  $B_{\pG_2}=\Spec\left(\CC[\NE(Y_{\pG_2})]\right)$.

\subsubsection{The toric base}
The base variety $B_{\pG_2}$ is a singular affine toric variety with 40 equations of the form $z^X=z^Y$, where $X=Y$ is a linear relation for some classes $X,Y\in\NE(Y_{\pG_2})$. We only write down eight of the 40 equations, which will be relevant to our calculations:
\begin{equation}\label{eqns!G2-toric-base}
\begin{aligned} \relax
[D_i]+3[E_i]+2[F_{ik}]+[F_{jl}] &= [D_j]+3[D_{jk}]+2[D_k]+3[D_{kl}]+[D_l] \\ \relax
[D_{ij}]+[E_{ij}] &= [D_{jk}]+[D_k]+2[D_{kl}]+[D_l]+[D_{li}] 
\end{aligned}
\end{equation}


\subsubsection{The mirror family}
In this case $\sX_{\pG_2}/B_{\pG_2}$ is given by 20 relative equations, determined by eight tag equations:
\begin{align*}
\theta_i\theta_j &= z^{D_{ij}}\left(\theta_{ij} + z^{E_{ij}}\right) \\
\theta_{ij}\theta_{jk} &= z^{D_j}\left(\theta_j + z^{E_j}\right)\left(\theta_j + z^{E_j+F_{jl}}\right)\left(\theta_j + z^{E_j+F_{jl}+F_{ik}}\right)
\end{align*}
where the monomials appearing in the equations are counting certain classes of rational curves on $Y_{\pG_2}$. More precisely, the first tag equation is derived from $\theta_i\theta_j=\sum_{m=1}^2z^{[\Sigma_m]}\theta_{ij}^{-D_{ij}\cdot\Sigma_m}$, where $[\Sigma_1]=[D_{ij}]$ and $[\Sigma_2]=[D_{ij}]+[E_{ij}]$ are the two classes of an effective rational curve $\Sigma\subset Y_{\pG_2}$ such that $\Sigma \cdot D_i=\Sigma \cdot D_j =1$, and $\Sigma\cdot D'=0$ for all other irreducible components in the boundary $D'\subset D$. The second tag equation comes from $\theta_{ij}\theta_{jk}=\sum_{m}z^{[\Sigma_m]}\theta_{j}^{-D_j\cdot\Sigma_m}$, where 
$[\Sigma_m]$ runs over the classes of an effective rational curve $\Sigma\subset Y_{\pG_2}$ such that $\Sigma \cdot D_{ij}=\Sigma \cdot D_{jk} =1$, and $\Sigma\cdot D'=0$ for all other irreducible components in the boundary $D'\subset D$.
As before, the other 14 equations defining $\sX_{\pG_2}$ can be found from the tag equations by working birationally. 

As explained in Remark~\ref{rmk!no-proper-proof}, to give a rigorous proof that this description holds we could use the machinery of scattering diagrams. However, we skip this since we are only interested in the existence of $X_{\pG_2}$ in order for us to use it as a key variety.

\subsubsection{The cluster variety} 
We write $N_1(Y)=\bm{D}\oplus\ZZ\langle\delta_{13},\delta_{24}\rangle$, according to the $\QQ$-basis $[D_1],\ldots,[D_{41}]$, $\delta_{13}$, $\delta_{24}$, where:
\[ \delta_{ik} = \tfrac13\left(2[D_i]+[D_j]+2[D_k]+[D_l]-2[F_{ik}]-[F_{jl}]\right) \]
The $\delta_{ik}$ were chosen so that
\begin{align*}
[E_i] &= [D_{jk}]+[D_{kl}]-[D_i]+\delta_{ik}, \\
[E_{ij}] &= [D_{jk}]+[D_k]+2[D_{kl}]+[D_l]+[D_{li}]-[D_{ij}],
\end{align*}
and therefore we can eliminate all of the coefficients $z^{E_i}$, $z^{E_{ij}}$ in a $\Dih_8$-equivariant way. Writing $A_i:=z^{D_i}$, $A_{ij}:=z^{D_{ij}}$, $\lambda_{ik}:=z^{\delta_{ik}}\left(1 + z^{F_{ik}} + z^{F_{ik}+F_{jl}}\right)$, and noting that
\[ 2\delta_{ik} = \delta_{jl}+[D_i]+[D_k]-[F_{ik}], \]
we recover our desired equations~\eqref{eqns!G2-format}. 
Since all powers are positive and integral we can easily extend all these equations to get an irreducible affine Gorenstein variety $X_{\pG_2}/\Aa^{10}$.

\subsubsection{Symmetries}
The cluster variety $X_{\pG_2}$ has the action of $\Dih_8\times\TT^8$, where $\Dih_8$ permutes the indices $\{1,2,3,4\}$. By calculating $D_i\cdot\delta_{ik}=1$ etc., we get the character table for the torus action $\TT^8\curvearrowright X_{\pG_2}$, as shown in Table \ref{table!char-G2}.
\begin{table}[h]
\caption{The character table for $\TT^D\curvearrowright X_{\pG_2}$.}\label{table!char-G2}
\begin{center}
\resizebox{\textwidth}{!}{
$\begin{array}{|c|cccccccc|cccccccccc|}\hline
 & \theta_1 & \theta_{12} & \theta_2 & \theta_{23} & \theta_3 & \theta_{34} & \theta_4 & \theta_{41} & A_1 & A_{12} & A_2 & A_{23} & A_3 & A_{34} & A_4 & A_{41} & \lambda_{13} & \lambda_{24} \\\hline
\chi_1    &  1 & 0 & 0 & 0 & 0 & 0 & 0 & 0 &-3 & 1 & 0 & 0 & 0 & 0 & 0 & 1 &-2 &-1 \\
\chi_{12} &  0 & 1 & 0 & 0 & 0 & 0 & 0 & 0 & 1 &-1 & 1 & 0 & 0 & 0 & 0 & 0 & 1 & 1 \\
\chi_2    &  0 & 0 & 1 & 0 & 0 & 0 & 0 & 0 & 0 & 1 &-3 & 1 & 0 & 0 & 0 & 0 &-1 &-2 \\
\chi_{23} &  0 & 0 & 0 & 1 & 0 & 0 & 0 & 0 & 0 & 0 & 1 &-1 & 1 & 0 & 0 & 0 & 1 & 1 \\
\chi_3    &  0 & 0 & 0 & 0 & 1 & 0 & 0 & 0 & 0 & 0 & 0 & 1 &-3 & 1 & 0 & 0 &-2 &-1 \\
\chi_{34} &  0 & 0 & 0 & 0 & 0 & 1 & 0 & 0 & 0 & 0 & 0 & 0 & 1 &-1 & 1 & 0 & 1 & 1 \\
\chi_4    &  0 & 0 & 0 & 0 & 0 & 0 & 1 & 0 & 0 & 0 & 0 & 0 & 0 & 1 &-3 & 1 &-1 &-2 \\
\chi_{41} &  0 & 0 & 0 & 0 & 0 & 0 & 0 & 1 & 1 & 0 & 0 & 0 & 0 & 0 & 1 &-1 & 1 & 1 \\\hline
\end{array}$}
\end{center}
\label{table!weight_table_for_G2}
\end{table}%

\subsection{Alternative formats for $X_{\pG_2}$}
We discuss some of the possible formats and useful subformats for $X_{\pG_2}$.

\subsubsection{Quadruple unprojection structure.}
Eliminating $\theta_{12},\theta_{23},\theta_{34},\theta_{41}$ from $\sA_{\pG_2}$ gives a Gorenstein projection $X_{\pG_2}\dashrightarrow Z$ where $Z$ is the following complete intersection of codimension 2:
\begin{align*} 
\theta_1\theta_3 &= A_{12}A_2A_{23}\theta_2 + \lambda_{24}A_{12}A_{23}A_{34}A_{41} + A_{34}A_4A_{41}\theta_4 \\
\theta_2\theta_4 &= A_{41}A_1A_{12}\theta_1 + \lambda_{13}A_{12}A_{23}A_{34}A_{41} + A_{23}A_3A_{34}\theta_3 
\end{align*}
This variety $Z$ is a family of affine surfaces over 
$\Aa^{10}_{A_i,A_{ij},\lambda_{ik}}$ whose general member has a compactification 
to a singular Del Pezzo surface with four lines at infinity meeting at four 
$\tfrac13(1,1)$ singularities, see \cite[\S2.2.1]{ch}. These four 
$\tfrac13(1,1)$ singularities are obtained by contracting the four $(-3)$-curves 
in $Y_{\pG_2}$. Each variable $\theta_{ij}$ can be recovered from $Z$ as a 
serial Gorenstein type I unprojection of the divisor 
$D_{ij}=V(A_{ij},\theta_k,\theta_l)$, giving the following codimension 3 
Pfaffian format $\Pf_4(M_{ij})=0$, where $M_{ij}$ is the matrix:
\begin{equation}\label{eq!G2-jerry}
 M_{ij} = \begin{pmatrix}
A_{jk}A_kA_{kl} & \theta_l & \theta_i & A_{ij} \\
& A_{li}(A_i\theta_i+\lambda_{ik}A_{jk}A_{kl}) & \theta_{ij} & \theta_j \\
& & A_{jk}(A_j\theta_j+\lambda_{jl}A_{kl}A_{li}) & \theta_k \\
& & & A_{kl}A_lA_{li}
\end{pmatrix}
\end{equation}
We note that $M_{ij}$ contains three unprojection divisors:
\begin{enumerate}
\item the Tom$_5$ ideal $(\theta_l,\theta_i,\theta_{ij},A_{jk})$ for the unprojection variable $\theta_{jk}$,
\item the Jer$_{24}$ ideal $(\theta_i,\theta_{ij},\theta_j,A_{kl})$ for the unprojection variable $\theta_{kl}$,
\item the Tom$_1$ ideal $(\theta_{ij},\theta_j,\theta_k,A_{li})$ for the unprojection variable $\theta_{li}$.
\end{enumerate} 
Taken altogether these variables give an unprojection cascade, partly shown in Figure~\ref{figure!g2-cascade}. 

\subsubsection{The $\pG_2^{(5)}$ and $\pG_2^{(4)}$ subformats}\label{sec!G2-subformats}
It is clear from equations \eqref{eqns!G2-format} that if $\phi^*(A_{ij})=1$ for some regular pullback $\phi$ from $X_{\pG_2}$, then the variable $\theta_{ij}$ becomes \emph{redundant}.
\begin{defn}
We define the \emph{$\pG_2^{(5)}$ format of codimension~5} by making the specialisation $A_{12}=1$ and eliminating the redundant variable $\theta_{12}$. We define the \emph{$\pG_2^{(4)}$ format of codimension~4} by making the specialisation $A_{12}=A_{34}=1$ and eliminating the redundant variables $\theta_{12}$ and~$\theta_{34}$.\footnote{These could also be obtained in a fancy way, by considering the mirror family for a log Calabi--Yau surface $(Y,D)$ whose anticanonical cycle has negative intersection degrees $(2,2,1,3,1,3,1)$ or $(2,2,1,2,2,1)$ respectively.}
\end{defn}

\paragraph{$\pG_2^{(5)}$ format} This is a \emph{``triple Jerry''} format. In terms of the matrices $M_{ij}$ defined above, the 14 equations are:
\begin{gather*}
\Pf_4\Big(M_{23}|_{A_{12}=1}\Big)=0,\quad  \Pf_4\Big(M_{34}|_{A_{12}=1}\Big)=0,\quad  \Pf_4\Big(M_{41}|_{A_{12}=1}\Big)=0, \\
\theta_{23}\theta_{34} = (\text{long equation}),\quad \theta_{34}\theta_{41} = (\text{long equation}), \quad \theta_{41}\theta_{12} = (\text{long equation}).
\end{gather*}
If we wish to keep the variable $\theta_{12}$ with the equation $\theta_{12}= \theta_1\theta_2 - A_{23}A_3A_{34}^2A_4A_{41}$, then we call this $\pG_2^{(5)}+\CI^{(1)}$ format.

\paragraph{$\pG_2^{(4)}$ format} This is a \emph{``double Jerry''} format (cf.~\cite[\S9]{bkr}). The 9 equations are:
\[
\Pf_4(M_{23}|_{A_{12}=A_{34}=1})=0, \quad  \Pf_4(M_{41}|_{A_{12}=A_{34}=1})=0, \quad \theta_{23}\theta_{41}=(\text{long equation}).
\]
If we wish to keep the variables $\theta_{12},\theta_{34}$ and their tag equations, then we call this $\pG_2^{(4)}+\CI^{(2)}$ format.

\subsection{Affine pieces and the deep locus} \label{sect!G2-deep}

Similarly to the $\pC_2$ cluster variety $X_{\pC_2}$, the $\pG_2$ cluster variety $X_{\pG_2}$ is partly covered by eight affine $\CI^{(6)}$ charts where one of each of the cluster variables $\theta_i$ or $\theta_{ij}$ does not vanish. The deep locus $X_0=X_{\pG_2}\cap V(\theta_1,\ldots,\theta_{41})$ breaks up into the following 28 linear subvarieties. 
\begin{align*}
\tag{$\times2$} \Aa^8 &\cong V(\theta_1,\ldots,\theta_{41},A_{ij},A_{kl}) \\
\tag{$\times8$} \Aa^7 &\cong V(\theta_1,\ldots,\theta_{41},A_i,A_{ij},A_{jk}) \\
\tag{$\times8$} \Aa^7 &\cong V(\theta_1,\ldots,\theta_{41},A_i,A_{jk},\lambda_{ik}) \\
\tag{$\times4$} \Aa^7 &\cong V(\theta_1,\ldots,\theta_{41},A_i,A_{jk},A_{kl}) \\
\tag{$\times4$} \Aa^7 &\cong V(\theta_1,\ldots,\theta_{41},A_{ij},A_k,A_l) \\
\tag{$\times2$} \Aa^6 &\cong V(\theta_1,\ldots,\theta_{41},A_i,A_k,\lambda_{ik},\lambda_{jl}) 
\end{align*}

\subsection{Regular pullbacks from $\pG_2$ format}
\label{eg!first-G2} 

Let $(X_{\pG_2},\mu,\FF)$ be a $\pG_2$ cluster format determined by the one parameter subgroup 
\[ \rho=(\rho_1,\rho_{12},\dots,\rho_4,\rho_{41})\colon\CC^*\to\TT^D. \]
From Table~\ref{table!weight_table_for_G2}:
\begin{gather*}
d(\theta_i)=\rho_{i},\quad d(\theta_{ij})=\rho_{ij},\quad d(A_i)=\rho_{li}-3\rho_i+\rho_{ij}, \quad d(A_{ij})=\rho_i-\rho_{ij}+\rho_j,\\
d(\lambda_{ik})= - 2\rho_i + \rho_{ij} - \rho_j + \rho_{jk} - 2\rho_k + \rho_{kl} - \rho_l + \rho_{li}.
\end{gather*}
We use the following shorthand to write down a regular pullback from $\pG_2$ format:
\[ \pG_2 \left(\begin{array}{cccc|cccc|cc}
\phi^*(\theta_{12}) & \phi^*(\theta_{23}) & \phi^*(\theta_{34}) & \phi^*(\theta_{41}) & \phi^*(\theta_1) & \phi^*(\theta_2) & \phi^*(\theta_3) & \phi^*(\theta_4) & \phi^*(\lambda_{13}) \\
\phi^*(A_{12}) & \phi^*(A_{23}) & \phi^*(A_{34}) & \phi^*(A_{41}) & \phi^*(A_1) & \phi^*(A_2) & \phi^*(A_3) & \phi^*(A_4) & \phi^*(\lambda_{24})
\end{array} \right) \]
or the same array with integer entries if we just wish to denote the degrees.

As with the regular pullbacks from $X_{\pC_2}$, it is easy to use the $M$-graded Hilbert series of $X_{\pG_2}$ to get the $\ZZ$-graded Hilbert series of $(X_{\pG_2},\mu,\FF)$. Again, the Hilbert numerator has adjunction number $\alpha=\rho_1+\rho_{12}+\rho_2+\rho_{23}+\rho_3+\rho_{34}+\rho_4+\rho_{41}$.

\subsection{Singular locus} As we did with the $\pC_2$ cluster variety $X_{\pC_2}$ we now describe the singular locus of $X_{\pG_2}$ and some of the hyperplane sections of $X_{\pG_2}$.

\begin{lem}\label{lemma!G2_singular_locus}
The reduced singular locus $\sing\left(X_{\pG_2}\right)$ is contained inside the deep locus
\[ \sing\left(X_{\pG_2}\right) \subset X_0 := X_{\pG_2}\cap V(\theta_1,\theta_{12},\theta_2,\theta_{23},\theta_3,\theta_{34},\theta_4,\theta_{41}) \]
and decomposes into 14 irreducible linear subvarieties, given by:
\begin{align*}
\Aa^8 &= X_0\cap V(A_{ij},A_{kl}), &&& \Aa^7 &= X_0\cap V(A_i,A_{jk},A_{kl}), \\
\Aa^5 &= X_0\cap V(A_i,A_{ij},A_{jk},A_k,\lambda_{ik}), &&& \Aa^5 &= X_0\cap V(A_i,A_j,A_{kl},\lambda_{ik},\lambda_{jl}). 
\end{align*}
In particular all components of the singular locus have codimension $\geq4$ in $X_{\pG_2}$.

The hyperplane section $X^z=X_{\pG_2}\cap V(z)$ is singular in codimension 1 if 
$z=\theta_i$ or $z=\theta_{ij}$.
In other cases,
\begin{enumerate}
\item $\sing\left(X^{A_i}\right)$ is contained in the locus $X^{A_i}_0 := V_{X^{A_i}}(\theta_{ij},\theta_j,\theta_{jk},\theta_k,\theta_{kl},\theta_l,\theta_{li})$ and contains the following components which have codimension $3$ in $X^{A_i}$:
\[ \Aa^8 = X^{A_i}_0\cap V(A_{jk},A_{kl}),  \quad
\Aa^8 = X^{A_i}_0\cap V(A_{jk},\lambda_{ik}) \quad \text{and} \quad
\Aa^8 = X^{A_i}_0\cap V(A_{kl},\lambda_{ik}), \]
\item $\sing\left(X^{A_{ij}}\right)$ is contained in the locus $X^{A_{ij}}_0 := V_{X^{A_{ij}}}(\theta_j,\theta_{jk},\theta_k,\theta_{kl},\theta_l,\theta_{li},\theta_i)$ and contains the following components which have codimension $\leq3$ in $X^{A_{ij}}$:
\[ \Aa^9=X^{A_{ij}}_0\cap V(A_{kl}) \quad \text{and} \quad 
\Aa^8=X^{A_{ij}}_0\cap V(A_{jk},A_{li}). \]
\end{enumerate}
\end{lem}


\begin{proof}
This is slightly more delicate than the computation of $\sing(X_{\pC_2})$, since asking the computer to compute the $6\times 6$ minors of the $18\times 20$ Jacobian matrix $J$ is fairly hopeless. First of all, if one of the cluster variables $\theta_i$ or $\theta_{ij}$ is nonzero we are in one of the affine complete intersection charts of Lemma \ref{lem!partial-covering} and it is easy to check that these are smooth. Therefore $\sing(X_{\pG_2})$ is contained in the deep locus $\sing(X_{\pG_2})\subseteq X_0$. Let $\Pi$ be one of the 28 irreducible components of $X_0$ listed in \S\ref{sect!G2-deep}, and take the restriction $J|_\Pi$. It turns out that $J|_{\Pi}$ is rather sparse, and it is then much easier to compute $\sing(X_{\pG_2})|_{\Pi}$ for each $\Pi$. Finally we take the union of all of these singular subloci and compute the irreducible components of this union. We see that $\sing(X_{\pG_2})$ has the 14 irreducible components above. 

The singular loci of $X^{A_i}$ and $X^{A_{ij}}$ can be computed in a similar way (with the appropriate adjustments to $X_0$), although it is easier just to check the inclusion of the components claimed in the statement of the Proposition directly.

If $\theta_{12}=0$, then $T=\frac{A_1\theta_1}{A_{23}A_{34}}$ is a solution to the monic polynomial equation 
\[ T^3 + \mu T^2 + \lambda A_1A_3 T + A_1^2A_2A_3^2A_4 = 0 \]
over the ring $\sO_{X^{\theta_{12}}}$, and hence $\sO_{X^{\theta_{12}}}$ is not integrally closed. Moreover, since $\theta_{12}$ is not a zero divisor in $\sA_{\pG_2}$ we know that $X^{\theta_{12}}$ is Gorenstein (hence $S_2$) and therefore must be singular in codimension 1. By a similar argument, because $U=\frac{A_{12}A_2\theta_2}{A_{34}}$ solves the monic equation
\[U^2 + \lambda_{24}A_{41}A_{12}U + \lambda_{13}A_2A_4A_{12}^2A_{41}^2 + A_2A_3A_4A_{12}A_{41}\theta_3=0 \]
over the ring $\sO_{X^{\theta_{1}}}$, $X^{\theta_1}$ is also singular in codimension 1.
\end{proof}

\subsection{Quasismoothness conditions}
Let $\widehat Y\subset \Aa^{10}$ be the 4-dimensional affine cone over a quasismooth weighted projective 3-fold $Y$.

\begin{prop}\label{proposition!G26_subformats}
Suppose that $\widehat Y = \phi^{-1}(X_{\pG_2})$ is a regular pullback and is not a complete intersection. Then, for any $i,j$, we must have $d(\theta_i)>0$, $d(\theta_{ij})>0$, and one of the following conditions must hold, up to $\Dih_8$ symmetry:
\begin{enumerate} \setlength\itemsep{0.5em}
\item {\bf $\pG_2^{(6)}$ format}: $d(A_{ij})>0$ for all $i,j$. Then $d(A_i)\ge0$ for all $i$ and $d(\lambda_{13}),d(\lambda_{24})\ge0$.

\item {\bf $\pG_2^{(5)}+\CI^{(1)}$ format}: $d(A_{12})=0$ and $d(A_{ij})>0$ for all other $i,j$. Then $d(A_1),d(A_2)\ge0$ and $d(\lambda_{13}),d(\lambda_{24})\ge0$. (See Corollary \ref{proposition!G25_subformats} for further analysis.)

\item {\bf $\pG_2^{(4)} + \CI^{(2)}$ format}: $d(A_{12})=d(A_{34})=0$ and $d(A_{23}),d(A_{41})>0$. (See Corollary \ref{proposition!G24_subformats} for further analysis.)

\item {\bf $\pC_2 + \CI^{(2)}$ format}: $d(A_{12})=d(A_{23})=0$ and $d(A_{34}),d(A_{41})>0$. (See Proposition \ref{proposition!C2_subformats}.)

\item {\bf $\pA_2 + \CI^{(3)}$ format}: $d(A_{12})=d(A_{23})=d(A_{34})=0$ and $d(A_{41})>0$.
\end{enumerate}
If all four $d(A_{ij})=0$ then we are in $\CI^{(6)}$ format. In other words, we consider cases according to the following cascade of specialisations (cf.\ Figure~\ref{figure!g2-cascade}, page \pageref{figure!g2-cascade}):
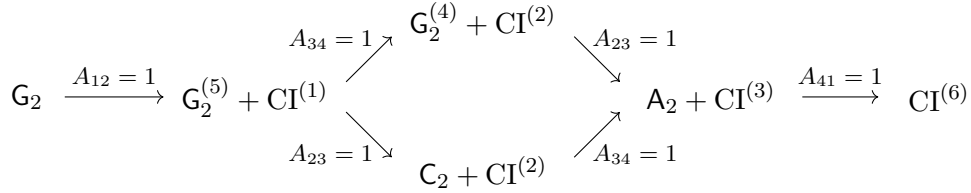
\begin{center}\begin{tikzpicture}
  \node at (0,1) {$\pG_2$};
  \node at (3,1) {$\pG_2^{(5)}+\CI^{(1)}$};
  \node at (6,2) {$\pG_2^{(4)}+\CI^{(2)}$};
  \node at (6,0) {$\pC_2+\CI^{(2)}$};
  \node at (9,1) {$\pA_2+\CI^{(3)}$};
  \node at (12,1) {$\CI^{(6)}$};
  \draw[->] (0.5,1) -- node [above] {\footnotesize $A_{12}=1$} (1.8,1);
  \draw[->] (4.2,1.2) -- node [above,xshift=-0.5cm] {\footnotesize $A_{34}=1$} (4.8,1.8);
  \draw[->] (4.2,0.8) -- node [below,xshift=-0.5cm] {\footnotesize $A_{23}=1$} (4.8,0.2);
  \draw[->] (7.2,1.8) -- node [above,xshift=0.5cm] {\footnotesize $A_{23}=1$} (7.8,1.2);
  \draw[->] (7.2,0.2) -- node [below,xshift=0.5cm] {\footnotesize $A_{34}=1$} (7.8,0.8);
  \draw[->] (10.2,1) -- node [above] {\footnotesize $A_{41}=1$} (11.2,1);
\end{tikzpicture}\end{center}
(\emph{Note}: After making the specialisation we do not automatically assume 
that the redundant variables are eliminated.)
\end{prop}

\begin{proof}
There are two things to prove. First that the claimed inequalities on the degrees are necessary in each case, and second that $\pC_2$ and $\pA_2$ format appear as claimed in case (4) and case (5) respectively.

We obtain the degree inequalities in each case by considering what happens if one of the variables is allowed to take a negative degree. Up to the $\Dih_8$ symmetry we can reduce to one of the following cases:

{\bf Claim 1}: $d(\theta_1)<0$ and $d(\theta_{12})<0$ cannot happen. 

Using Lemma \ref{lemma!G2_singular_locus}, we see that $d(\theta_1)\ge0$ and $d(\theta_{12})\ge0$. A single equality $d(\theta_1)=0$ or $d(\theta_{12})=0$ would reduce $\widehat Y$ to a complete intersection $\CI^{(6)}$ by Lemma \ref{lem!partial-covering}. Thus from now on, we assume $d(\theta_i)>0$, $d(\theta_{ij})>0$. 

{\bf Claim 2}: $d(A_{12})<0$ cannot happen. 

If $d(A_{12})<0$ then we must have $d(A_{34})=0$ and either $d(A_{23})=0$ or $d(A_{41})=0$ to avoid pulling back the two bad components of Lemma~\ref{lemma!G2_singular_locus}(2). This puts us in case (4), which is $\pC_2^{(4)}+\CI^{(2)}$ format. Proposition \ref{proposition!C2_subformats} combined with the coordinate change described below, implies that $d(A_{12})$ can not be negative. 

{\bf Claim 3}: $d(A_1)<0$ puts us in case (2).

If $d(A_1)<0$ we need either $d(A_{23})=0$ or $d(A_{34})=0$, which puts us in case (2). Assume the former case, then if one of $d(A_2)<0$ or $d(A_3)<0$, this forces one of $d(A_{34})=0$, $d(A_{41})=0$ or $d(A_{12})=0$, and so we are either in case (3) or case (4).

{\bf Claim 4}: $d(\lambda_{13})<0$ puts us in case (4). 

Table~\ref{table!weight_table_for_G2} can be used to obtain the following identities:
\begin{equation} \label{eqn!lambda-mu} \tag{$\dagger$} \begin{split}
3d(\lambda_{13}) &= 2d(A_1)+d(A_2)+2d(A_3)+d(A_4)\\ 
3d(\lambda_{24}) &= d(A_1)+2d(A_2)+d(A_3)+2d(A_4)\\
d(\lambda_{13})+d(\lambda_{24}) &= d(A_1)+d(A_2)+d(A_3)+d(A_4)
\end{split} \end{equation}
If $d(\lambda_{13})<0$ the first of these implies that $d(A_i)<0$ for some $i$---without loss of generality either $A_1$ or $A_2$. 

If $d(A_1)<0$ then, to avoid pulling back the big components of $\sing(X^{A_1})$, we need 
\[ \label{eqn!dBij} \tag{$\ddagger$} 
{\small \big(d(A_{23})=0 \text{ or }d(A_{34})=0\big)\text{ and }\big(d(A_{23})=0\text{ or }d(\lambda_{13})=0\big) \text{ and }\big(d(A_{34})=0\text{ or }d(\lambda_{13})=0\big).} \]
Since $d(\lambda_{13})<0$, this implies $d(A_{23})=d(A_{34})=0$ and we are in case (4).

If $d(A_2)<0$ then, to avoid pulling back the big components of $\sing(X^{A_2})$, we need
\[ \big(d(A_{34})=0 \text{ or }d(A_{41})=0\big)\text{ and }\big(d(A_{41})=0\text{ or }d(\lambda_{24})=0\big) \text{ and }\big(d(A_{34})=0\text{ or }d(\lambda_{24})=0\big). \]
If $d(A_{34})=d(A_{41})=0$ then we go to case (4), otherwise $d(\lambda_{24})=0$ and the relations \eqref{eqn!lambda-mu} imply $d(A_1)+d(A_3)=2d(\lambda_{13})<0$, so either $d(A_1)<0$ or $d(A_3)<0$. This again forces two consecutive $A_{ij}$ to have degree zero, and we go to case (4). 

This completes our rough analysis of the admissible degrees in cases (1)--(5). We now show that $\pC_2$ format and $\pA_2$ format appear in cases (4) and (5). By definition, case (2) is $\pG_2^{(5)}$ format and case (3) is $\pG_2^{(4)}$ format.

{\bf Case (4) is $\pC_2+\CI^{(2)}$ format}: Pulling back $X_{\pG_2}$ by the morphism $\phi\colon \Aa^{16}\to \Aa^{18}$ given by $\phi^*(A_{12}) = \phi^*(A_{23}) = 1$ and $\phi^*(z)=z$ for all other variables, gives a complete intersection of codimension two
\[ \theta_{12} = \theta_1\theta_2 - A_3A_{34}^2A_4A_{41}, \quad \theta_{23} = \theta_2\theta_3 - A_{34}A_4A_{41}^2A_1 \]
inside the following generic pullback from $\pC_2$ format:
\[ \pC_2\left(\begin{array}{ccc|ccc|}
 \theta_{41} & \theta_2 & \theta_{34} & \theta_1 & \theta_3 & \theta_4 \\
  A_{41} & A_2 & A_{34} & A_1 & A_3 & A_4\theta_4+\lambda_{24}
\end{array} \:\: \lambda_{13} \right), \]
i.e.\ inside $\Aa^2_{\theta_{12},\theta_{23}}\times X_{\pC_2}$. Under this coordinate change, the degrees of all variables must satisfy the conditions of Proposition \ref{proposition!C2_subformats}.


{\bf Case (5) is $\pA_2+\CI^{(3)}$ format}: Pulling back $X_{\pG_2}$ by the morphism $\phi\colon \Aa^{15}\to \Aa^{18}$ given by $\phi^*(A_{12}) = \phi^*(A_{23}) = \phi^*(A_{34}) = 1$ and $\phi^*(z)=z$ for all other variables, gives:
\[ \Pf_4\begin{pmatrix}
A_2 & \theta_3 & \theta_4 & A_{41} \\
& A_4\theta_4+\lambda_{24} & \theta_{41} & \theta_1 \\
& & A_1\theta_1+\lambda_{13} & \theta_2 \\
& & & A_3
\end{pmatrix} \qquad
 \begin{array}{rcl}
\theta_{12} \!\!&=&\!\! \theta_1\theta_2 - A_3A_4A_{41},\\
\theta_{23} \!\!&=&\!\! \theta_2\theta_3 - A_4A_{41}^2A_1,\\
\theta_{34} \!\!&=&\!\! \theta_3\theta_4 - A_{41}A_1A_2.
\end{array}\]
If $\widehat Y$ is quasismooth and not a complete intersection, then the entries of 
this matrix must all have positive degrees \cite[Proposition 2.7]{bkz}.
\end{proof}

\subsubsection{Some further subformats}
Let $\widehat Y \subset \Aa^{10}$ be the 4-dimensional affine cone over a 
quasismooth weighted projective 3-fold $Y$. We refine cases (2) and (3) of Proposition \ref{proposition!G26_subformats}:

\begin{cor}[Subformats for $\pG_2^{(5)}$]\label{proposition!G25_subformats}
Suppose that $\widehat Y$ is in $\pG_2^{(5)}$ format, i.e.\ $d(A_{12})=0$, all other $d(A_{ij})>0$, $d(A_1),d(A_2)\ge0$ and $d(\lambda_{13}),d(\lambda_{24})\ge0$. There are three possibilities:
\begin{enumerate}
 \item $d(A_3),d(A_4)\ge0$.
 \item $d(A_3)<0$ and $d(\lambda_{13})=0$.
 \item $d(A_3),d(A_4)<0$ and $d(\lambda_{13})=d(\lambda_{24})=0$.
\end{enumerate}
\end{cor}
\begin{proof}
From equation \eqref{eqn!dBij} and its translates under $\Dih_8$, if $d(A_3)<0$ then $d(\lambda_{13})=0$, and if $d(A_4)<0$ then $d(\lambda_{24})=0$. This completes the proof.
\end{proof}
\begin{rmk} We do not have special formats for cases (2) and (3), but their divisor class group has rank $>1$.
\end{rmk}

\begin{cor}[Subformats for $\pG_2^{(4)}$]\label{proposition!G24_subformats}
Suppose that $\widehat Y$ is in $\pG_2^{(4)}$ format, that is, $d(A_{12})=d(A_{34})=0$ and $d(A_{23}),d(A_{41})>0$. There are three possibilities (up to symmetry):
\begin{enumerate}
\item If $d(\lambda_{13}),d(\lambda_{24})>0$, then $d(A_i)\ge0$ for all $i$ and $\widehat Y$ is in (strict) $\pG_2^{(4)}$ format.
\item If $d(\lambda_{24})=0$ and $d(\lambda_{13})>0$, then $\widehat Y$ is in rolling factors format.
\item If $d(\lambda_{13})=d(\lambda_{24})=0$, then $\widehat Y$ is in $\PP^1\times \PP^1\times \PP^1$ format.
\end{enumerate}
\end{cor}

\begin{proof}
We have that $d(A_{12})=d(A_{34})=0$, so we assume that $\phi^*(A_{12})=\phi^*(A_{34})=1$. By equation \eqref{eqn!dBij}, if $d(A_1)<0$ or $d(A_3)<0$ then we must have $d(\lambda_{13})=0$. Similarly, if $d(A_2)<0$ or $d(A_4)<0$ then we must have $d(\lambda_{24})=0$.

Suppose we are in case (2). Then equations \eqref{eqn!lambda-mu} combined with the above, imply that at least one of $d(A_2)<0$ or $d(A_4)<0$ is negative, and $d(A_1),d(A_3)\geq0$. Say $d(A_2)<0$ and let $\phi\colon\Aa^{14}\to\Aa^{18}$ be defined 
by $\phi^*(A_2)=0$, $\phi^*(\lambda_{24})=\phi^*(A_{12})=\phi^*(A_{34})=1$, and $\phi^*(z)=z$ otherwise. Since $d(\lambda_{24})=0$, the following coordinate changes are homogeneous
\[A_{23}' = A_4\theta_4 + A_{23}, \quad \theta_2' = \theta_2 + \lambda_{13}A_4A_{41} + A_3A_4\theta_3, \]
and the ideal defining $\phi^*X_{\pG_2}$ is in rolling factors format:
\[ \bigwedge^2\begin{pmatrix}
\theta_{23} & \theta_3 & \theta_4 & A_{41}  \\
\theta_2'  & A_{23}' & \theta_{41} & \theta_1  
\end{pmatrix} =0 \quad \quad \begin{array}{rcl}
\theta_{23}\theta_4 &=& A_3\theta_3^2 + \lambda_{13} A_{41}\theta_3 + A_1A_{41}^2 \\
\theta_{23}\theta_{41} &=& A_3A_{23}'\theta_3 + \lambda_{13} A_{41}A_{23}' + A_1A_{41}\theta_1 \\
\theta_2'\theta_{41} &=& A_3A_{23}'^2 + \lambda_{13} \theta_1A_{23}' + A_1\theta_1^2 \\
\end{array}\]

If we are in case (3), then equations \eqref{eqn!lambda-mu} reduce to $d(\lambda_{13})=d(A_1)+d(A_3)$ and $d(\lambda_{24})=d(A_2)+d(A_4)$. Thus $d(A_1)=-d(A_3)$ and $d(A_2)=-d(A_4)$. So either two consecutive 
$A_i$ have negative degree, or $d(A_i)=0$ for all $i$.

For the former case, suppose $\phi^*(A_2)=\phi^*(A_3)=0$, $\phi^*(\lambda_{13})=\phi^*(\lambda_{24})=\phi^*(A_{12})=\phi^*(A_{34})=1$ and $\phi^*(z)=z$ otherwise. Then $\phi^*(X_{\pG_2})\subset\Aa^{12}$ is in $\PP^1\times\PP^1\times\PP^1$ format, defined by the $2\times2$ minors of the following cube after the displayed coordinate changes:
\renewcommand{\arraystretch}{1.2}
\begin{center} \begin{tikzpicture}[scale=0.75]
   \draw (0,0) -- (2,0) -- (3,1) -- (3,3) -- (1,3) -- (0,2) -- cycle;
   \draw (0,0) -- (1,1) -- (1,3) (1,1) -- (3,1) (0,2) -- (2,2) -- (2,0) (2,2) -- (3,3);
   \node[fill=white] at (0,0) {$\theta_2'$};
   \node[fill=white] at (2,0) {$\theta_{23}$};
   \node[fill=white] at (1,1) {$\theta_{1}$};
   \node[fill=white] at (3,1) {$A_{41}$};
   \node[fill=white] at (0,2) {$A_{23}'$};
   \node[fill=white] at (2,2) {$\theta_3'$};
   \node[fill=white] at (1,3) {$\theta_{41}$};
   \node[fill=white] at (3,3) {$\theta_4$};
   \node at (9,1.5) {$\begin{array}{rcl} \theta_2' &=& \theta_2 + A_4A_{41} \\
   \theta_3' &=& \theta_3 + A_{41}A_1 \\
   A_{23}' &=& A_{23} + A_4\theta_4 + A_1\theta_1 \end{array}$};
\end{tikzpicture} \end{center}

In the latter case, $\phi^*(A_i)=\phi^*(\lambda_{13})=\phi^*(\lambda_{24})=\phi^*(A_{12})=\phi^*(A_{34})=1$ for all 
$A_i$, and $\phi^*(z)=z$ otherwise. Then $\phi^*(X_{\pG_2})\subset\Aa^{12}$ 
is in $\PP^1\times\PP^1\times\PP^1$ format, defined by the $2\times2$ minors of the cube after the displayed coordinate changes:
\renewcommand{\arraystretch}{1.2}
\begin{center} \begin{tikzpicture}[scale=0.75]
   \draw (0,0) -- (2,0) -- (3,1) -- (3,3) -- (1,3) -- (0,2) -- cycle;
   \draw (0,0) -- (1,1) -- (1,3) (1,1) -- (3,1) (0,2) -- (2,2) -- (2,0) (2,2) -- (3,3);
   \node[fill=white] at (0,0) {$\theta_2'$};
   \node[fill=white] at (2,0) {$\theta_{23}$};
   \node[fill=white] at (1,1) {$\theta_1'$};
   \node[fill=white] at (3,1) {$A_{41}'$};
   \node[fill=white] at (0,2) {$A_{23}'$};
   \node[fill=white] at (2,2) {$\theta_3'$};
   \node[fill=white] at (1,3) {$\theta_{41}$};
   \node[fill=white] at (3,3) {$\theta_4'$};
   \node at (9,1.5) {$\begin{array}{rcl} 
   \theta_2' &=& \theta_2+\theta_3 \\
   \theta_4' &=& \theta_4+\theta_1 \\
   \theta_1' &=& \epsilon\theta_1 + \epsilon^3\theta_4 + \sqrt{2}\epsilon^2A_{23} \\
   -A_{23}' &=& \epsilon^3\theta_1 + \epsilon\theta_4 + \sqrt{2}\epsilon^2A_{23} \\
   -\theta_3' &=& \epsilon\theta_3 + \epsilon^3\theta_2 + \sqrt{2}\epsilon^2A_{41} \\
   A_{41}' &=& \epsilon^3\theta_3 + \epsilon\theta_2 + \sqrt{2}\epsilon^2A_{41}
   \end{array} $};
\end{tikzpicture}  \end{center}
where $\epsilon$ is a primitive $8$th root of unity.
\end{proof}

\section{Applications to constructing Fano 3-folds}\label{sec!fanos}
\subsection{Introduction to Fano 3-folds}
A \emph{Fano 3-fold} is a normal projective 3-fold $Y$ with at worst $\QQ$-factorial terminal singularities and whose anticanonical divisor $-K_Y$ is $\QQ$-Cartier and ample. The \emph{Fano index} of $Y$ is the largest positive integer $q$, such that $-K_Y=qA$ for some ample Weil divisor $A$. If the Weil divisor class group $\Cl(Y)=\ZZ$, then $Y$ is called \emph{prime}. The discrete invariants of $Y$ are $q$, $h^0(Y,A)$ and the basket of terminal quotient singularities $\mathcal{B}$. There are a finite number of numerical possibilities for $(q,h^0(Y,A),\mathcal{B})$, and approximately 50,000 such are listed in \cite{grdb}, produced using \cite{abr,bs,bs2}. We refer to any one such numerical possibility as a \emph{candidate} Fano 3-fold.

The next stage of the classification is to prove whether a given candidate $Y$ exists, and then to investigate the structure of the Hilbert scheme of $Y$. We construct $Y$ by taking $\Proj$ of the finitely generated Gorenstein graded ring $R(Y,A)=\bigoplus_{n\ge0}H^0(Y,nA)$. A choice of generators for $R(Y,A)$ gives an embedding of $Y$ into weighted projective space $\PP(a_1,\dots,a_n)$. From now on, we assume that $Y$ is quasismooth with at worst terminal quotient singularities. The expected codimension of $Y$ may be computed from the Hilbert series $P_{(Y,A)}(t)=\sum_{n\ge0}h^0(Y,nA)t^n$, which is in turn computed using the above invariants. Since our cluster formats have codimension 4, 5 or 6, we only consider those candidates whose expected dimension lies in this range. We further assume that $R(Y,A)$ is generated as simply as possible; that is, we do not consider specialisations of $A$ (e.g.~hyperelliptic, trigonal, etc.), which may also have cluster format constructions, but in higher than expected codimension.

\subsection{Primality of Fano 3-folds}\label{sec!primality}
We give a criterion for checking primality of quasismooth varieties in cluster format.
\begin{lem}
Let $k$ be an algebraically closed field of characteristic $0$. Suppose that $Y=\phi^{-1}(X)$ is a quasismooth variety defined over $k$, of dimension $\ge3$ in cluster format. Choose one of the cluster variables $\theta_i$ or $\theta_{ij}$ and denote it by $\theta$. If $\phi^*(\theta)$ is a prime element of $k[Y]$, then every Weil divisor on $Y$ is of the form $\sO_Y(n)$ for some $n$.
\end{lem}
\begin{proof} We show that the coordinate ring $k[\widehat Y]$ is factorial. If $\tau=\phi^*(\theta)$ is a prime element of $k[\widehat Y]$, then by Nagata's lemma \cite[Theorem 20.2]{Matsumura}, it suffices to show that the localisation $k[\widehat Y]_{\tau}$ is factorial. The open set $\widehat Y\cap(\tau\ne0)$ is a complete intersection, because the localisation at $\tau$ factors through the open subset $\widehat X\cap(\theta\ne0)$, which is a complete intersection by Lemma \ref{lem!partial-covering}. Since complete intersections of dimension $\ge4$ are parafactorial and $\widehat Y$ is regular outside the vertex (by quasismoothness), it follows that $k[\widehat Y]_{\tau}$ is factorial (see \cite[XI 3.10, 3.13]{sga2}).
\end{proof}
The following theorem summarises the application of this criterion to our list of Fano 3-folds in cluster formats:
\begin{thm} \leavevmode
\begin{enumerate}
\item If $Y$ is in $\pC_2$ format and not $\PP^2\times\PP^2$ subformat, then $Y$ is prime;
\item If $Y$ is in $\pG_2^{(4)}$ format and not rolling factors or $(\PP^1)^3$ subformat, then $Y$ is prime;
\item If $Y$ is in $\pG_2^{(5)}$ format, and in case 1 of Corollary \ref{proposition!G25_subformats}, then $Y$ is prime. 
\end{enumerate}
\end{thm}
\begin{proof}
Primality depends on the format and on $\phi$, so we apply the above Lemma to each construction individually, using the computer. We do not check primality of $\tau=\phi^*(\theta)$ in $k[\widehat Y]$ directly, as the computer does this over $\QQ$, and $\tau$ could still be nonprime over $\CC$. 
Instead, we check that $Y\cap V(\tau)$ is nonsingular in codimension 1 (this computation is valid over $\CC$). Since $R_1+S_2$ is equivalent to normality, the fact that $Y$ is Gorenstein implies that $Y\cap V(\tau)$ is normal over $\CC$ and hence geometrically normal. Thus by \cite[IV \S4.6]{EGAIVb}, $Y\cap V(\tau)$ is geometrically irreducible, in particular irreducible over $\CC$.

Moreover, it follows from Lemmas \ref{lemma!C2_singular_locus} and \ref{lemma!G2_singular_locus}, that $Y\cap V(\theta)$ is necessarily singular in co\-dimension 1 for certain choices of $\theta$. Thus for $\pC_2$ format, we need only check
$\theta_1$, $\theta_2$, $\theta_3$, for $\pG_2^{(4)}$ format only $\theta_1$, $\theta_3$, and for $\pG_2^{(5)}$ format only $\theta_3$.
\end{proof}

\subsection{Comparison with Tom \& Jerry}
In this subsection, we suppose that $Y$ is a Fano 3-fold in codimension 4 with a type I centre. The definitive guide to this situation is \cite{bkr}, according to which, each $Y$ has at least two constructions: one Tom and one Jerry. In total, 274 of the 322 families from \cite{bkr} contain a subfamily which is in a cluster format. 

Based on analysis of our classification \cite{bigtables}, we make the following observation:
\begin{equation}\label{ass!TJ}
\tag{TJ} \parbox{\dimexpr\linewidth-4em}{Up to symmetry of the cluster format and choice of coordinates, the type I centre is positioned at the coordinate point $P_{\phi^*(\theta_{12})}$.}
\end{equation}
Thus if $Y$ is in $\pC_2$ format, then the projection is the $\text{Tom}_3$ matrix \eqref{eq!C2-tom}, and if 
$Y$ is in $\pG_2^{(4)}$ format, then the projection is the $\text{Jerry}_{24}$ matrix \eqref{eq!G2-jerry}.

Under assumption \eqref{ass!TJ}, we can transform the output of \cite{bkr} into a short list of possible cluster formats for $Y$, by permuting the row-columns of the skew-symmetric weight matrix appropriately. We work through a representative example.

\begin{eg} According to \cite{bkr}, candidate $\#5000$ $Y\subset\PP(1,1,3,4,4,5,5,9)$ has $\text{Tom}_4$ and $\text{Jerry}_{24}$ projections from the type $\rm I$ centre $\frac19(1,4,5)$, leading to an unprojection divisor $\PP(1,4,5)$ inside a Fano 3-fold $\overline Y\subset\PP(1,1,3,4,4,5,5)$ defined by the Pfaffians of a $5\times5$ skew matrix. We assume that $\overline Y$ is a Tom$_4$. The weights of this skew matrix $(m_{ij})$ are then
\[(m_{ij})=\begin{pmatrix}
3 & 4 & 3 & 4\\
& 5 & 4 & 5 \\
& & 5 & 6 \\
& & & 5
\end{pmatrix}\]
after swapping row-columns 3 and 4, to match up with \eqref{eq!C2-tom}.
According to \eqref{eq!C2-tom}, the 1-parameter subgroup $\rho\colon\CC^*\to\TT^D$ (see \ref{eg!first-C2}) corresponding to $(m_{ij})$ is 
$\rho=(d(\theta_{12}),m_{14},m_{25},m_{35},m_{13},m_{24})$. Further permutations fixing row--column 3 lead to different possibilities for $\rho$. After removing those which are invalid according to Proposition \ref{proposition!C2_subformats}, we get four possible $\pC_2$-formats matching $\text{Tom}_4$, indexed by the corresponding permutation:
\[
\rho=(9,3,5,6,4,4),\ \rho_{(1,2)}=(9,4,5,6,5,3),\ \rho_{(4,5)}=(9,5,4,5,4,5),\ \rho_{(1,2)(4,5)}=(9,5,3,5,5,5)
\]
Of these, $\rho\mapsto
\pC_2\left(\begin{smallmatrix}9 & 3 & 5 \\ 1 & 5 & 5 \end{smallmatrix}\big\vert
\begin{smallmatrix}6 & 4 & 4 \\ 2 & 4 & 0 \end{smallmatrix}\big\vert
\begin{smallmatrix}3\end{smallmatrix}\right)$ gives a working construction for $Y$, corresponding to a subfamily of that constructed by \cite{bkr}. The other three fail because the adjunction number is wrong.
%
We carried out a similar analysis for $\text{Jerry}_{24}$. There is no $\pG_2^{(4)}$ construction for candidate $\#5000$.
\end{eg}
Thus cluster format constructions do not exist for some of the families constructed by \cite{bkr}. Heuristically, the cluster format restricts the monomials available to the $5\times 5$ matrix, and this sometimes imposes worse than allowed singularities on $\overline Y$ and therefore $Y$.

\subsection{Fano 3-folds of large Fano index}

Table \ref{tab!fano-index} presents the data of \cite{grdb} for prime Fano 3-folds with $q\ge2$ in codimension 4, and its refinement using results of Prokhorov \cite[\S1]{P13} on the nonexistence of certain candidates. The last row lists our cluster format constructions which are prime.
\begin{table}[ht]
\caption{Fano 3-folds with $q\ge2$.}\label{tab!fano-index}
\[\renewcommand{\arraystretch}{1.1}
\begin{array}{lccccccc}
\text{Fano index $q$} & 2 & 3 & 4 & 5 & 6 & 7 & >7 \\
\hline
\text{GRDB candidates in codimension 4} & 37 & 11 & 5 & 2 & 3 & 3 & 0 \\
\text{candidates which do not exist} & ? & 1 & 1 & 0 & 1 & 1 & 0 \\
\text{candidates with cluster format constructions} & 27 & 8 & 4 & 2 & 2 & 2 & 0
\end{array}\]
\end{table}

Thus the question of existence is now settled in codimension 4 and Fano index $\ge4$. In particular, the constructions of two index 7 candidates provide an answer to a question of Prokhorov \cite[\S1.4]{P16}. For index 3, the missing candidates are $\#41058$ and $\#41245$. It would be interesting to know whether these exist. Brown and Suzuki \cite{bs} constructed 33 of the index 2 candidates, although it is not clear to us whether these constructions are prime. We have prime cluster format constructions corresponding to 27 of these 33 candidates. Thus there remain at least four candidates for which is it not known whether there is a prime construction, hence the ``?'' in the table.

The finer question of describing the Hilbert scheme for each of the candidates with $q\ge2$ remains open. For some candidates, we get two distinct cluster constructions, often both prime. Perhaps the general phenomenon from \cite{bkr} persists, and there are always at least two components to the Hilbert scheme, if we relax the requirement that $Y$ be prime.

\subsection{Fano 3-folds with empty $|{-K_Y}|$}

We use $\pC_2$ cluster format to construct two codimension 4 candidates with $|{-K_Y}|$ empty. These both have extrasymmetric descriptions induced by the $\pC_2$ format. We explain $\#25$ in some detail; $\#38$ is rather similar.
\begin{eg}\label{eg!25-continued}
Candidate $\#25$ is $Y\subset\PP(2,5,6,7,8,9,10,11)$. Let $p,q,r,s,t,u,v,w$ be coordinates on the ambient space. 
With the notation established in Example \ref{eg!first-C2}, the cluster format is $\pC_2\left(\begin{smallmatrix}8 & 10 & 12 \\ 8 & 7 & 9 \end{smallmatrix}\big\vert
\begin{smallmatrix}10 & 6 & 11 \\ 0 & 6 & 0 \end{smallmatrix}\big\vert
\begin{smallmatrix}3\end{smallmatrix}\right)$, and after coordinate choices, the general morphism $\phi\colon\Aa^8\to\Aa^{13}$ of degree~0 is:
\[\left(\begin{array}{ccc|ccc|}
R_8 & v & P_{12} & Q_{10} & S_6 & w \\
t & s & u & 1 & r & 1\end{array}
\:\: 0 \right),
\]
where $P_{12},Q_{10},R_8,S_6$ are general weighted homogeneous forms of degree given by the subscript.
Since $d(\lambda)=3$ forces $\phi^*(\lambda)=0$ for degree reasons, and $\phi^*(B_{23})=\phi^*(B_{31})=1$, the equations defining $Y$ have a nice extrasymmetric format with floating factor $r$ (see \ref{section!crazy}):
\[ 
\Pf_4\begin{pmatrix}
t & S_6 & v & w & u \\
& s & w & P_{12} & Q_{10} \\
& & u & Q_{10} & R_8 \\
& & & rt & rS_6 \\
& & & & rs
\end{pmatrix}. \]
\end{eg}
The third codimension 4 candidate with $|{-K_Y}|$ empty, $\#166$, does not have a cluster format construction. Indeed, a proposed construction for $\#166$ is as a $\ZZ/2$-quotient of a complete intersection Fano 3-fold \cite{AR}. This proposed construction has expected embedding codimension~$>4$. There are a handful of further candidates with $|{-K_Y}|$ empty in codimension 5 and 6, but none of these have cluster format constructions.

\subsection{Fano 3-folds with no projections}\label{section!no-projections}

According to \cite{AO} the candidates that are most likely to give rise to birationally rigid Fano 3-folds, are those with no centres of projection. In codimension 4 and Fano index 1, there are five such candidates, of which we construct three: $\#25$ has $|{-K_Y}|$ empty, and is treated above; $\#29374$ is a del Pezzo 3-fold, classically known; $\#282$ has two constructions, which we describe here:
\begin{eg} Candidate $\#282$ is $Y\subset\PP(1,6,6,7,8,9,10,11)$. Let $p,q,r,s,t,u,v,w$ be coordinates on the ambient space.
We first consider the cluster format 
$\pG_2\left(\begin{smallmatrix}15 & 9 & 21 & 12 \\ 0 & 7 & 0 & 8 \end{smallmatrix}\big\rvert
\begin{smallmatrix} 9 & 6 & 10 & 11 \\ 0 & 6 & 0 & 0 \end{smallmatrix}\big\vert
\begin{smallmatrix} 2 \\ 4 \end{smallmatrix}\right)$, which is in $\pG_2^{(4)}$ subformat, because $d(A_{12})=d(A_{34})=0$ (see \S\ref{eg!first-G2} for notation). The general morphism $\phi\colon\Aa^8\to\Aa^{18}$ is
\[\left(\begin{array}{cccc|cccc|c}
\theta_{12} & Q_9 & \theta_{34} & P_{12} & u & q & v & w & p^2 \\
1 & s & 1 & t & 1 & r & 1 & 1 & p^4 
\end{array}\right),\]
where the redundant variables $\theta_{12}$ and $\theta_{34}$ are eliminated, so we do not consider their images. The equations of $Y$ can be expressed as a double Jerry format following \S\ref{sec!G2-subformats}:
\begin{gather*}
\Pf_4\begin{pmatrix}t & u & q & s \\ & qr+p^4t & Q_9 & v \\ && v+p^2t & w \\ &&& t \end{pmatrix},\quad
\Pf_4\begin{pmatrix}rs & v & w & t \\ & w+p^4s & P_{12} & u \\ && u+p^2s & q \\ &&& s \end{pmatrix}\\
P_{12}Q_9=vw+p^4qw+p^2uv+uqr+str-stp^6.
\end{gather*}

The other construction of $\#282$ uses exactly the same $\pC_2$ cluster format as $\#25$ above (Example \ref{eg!25-continued}), giving an extrasymmetric construction whose explication we leave to the reader.

\end{eg}

\subsection{Fano 3-folds in codimensions 5}
There are 50 Fano 3-folds in $\pG_2^{(5)}$ format. Three of these have Fano index 3, one has index 2 and the the remainder have index 1. Moreover, all of the index 1 candidates that we construct in codimension 5, have type I centres.
\begin{eg} Consider the index 3 candidate $\#41117$ given by $Y\subset\PP(1,2,3,3,4,5,5,6,7)$ with coordinates $p,q,r,s,t,u,v,w,x$. From \cite{bigtables}, $Y$ has a $\pG_2\left(\begin{smallmatrix}11 & 4 & 5 & 7 \\ 0 & 3 & 1 & 3 \end{smallmatrix}\big\vert\begin{smallmatrix}6 & 5 & 2 & 4 \\ 0 & 0 & 3 & 0 \end{smallmatrix}\big\vert\begin{smallmatrix}2 \\ 1\end{smallmatrix}\right)$ cluster format construction. The general morphism $\phi\colon\Aa^9\to\Aa^{18}$ is
\[\left(\begin{array}{cccc|cccc|c}
\theta_{12} & P_4 & u & x & w & v & q & t & R_2 \\
1 & r & p & s & 1 & 1 & Q_3 & 1 & S_1 
\end{array}\right),\]
where $P_4,Q_3,R_2,S_1$ are general polynomials of degree denoted by the subscript, and the redundant variable $\theta_{12}$ is eliminated. Following \S\ref{sec!G2-subformats}, $Y$ is a triple Jerry format with 14 equations:
\begin{gather*}
\Pf_4\begin{pmatrix} ps & w & v & r \\ & v+Sps & P & q \\ && p(Qq+Rs) & t \\ &&& s \end{pmatrix},\quad
\Pf_4\begin{pmatrix}s & v & q & p \\ & r(Qq+Rs) & u & t \\ && s(t+Sr) & w \\ &&& r \end{pmatrix}\\
\Pf_4\begin{pmatrix}r & q & t & s \\ & p(t+Sr) & x & w \\ && w+Rrp & v \\ &&& Qrp \end{pmatrix},\quad
\begin{array}{c}
Pu= Qq^3+Rsq^2+Ss^2q+s^3 \\
ux= Qr^3+Rr^2t+Srt^2+t^3 \\
Px=(Qqt+(Q-RS)rs)p^2+p(Stv+Rqw)+vw
\end{array}
\end{gather*}

\end{eg}

\subsection{Why no Fano 3-folds in codimension 6 cluster format?} \label{subsection!no-G2-constructions}
There are no codimension 6 Fano 3-folds in $\pG_2^{(6)}$ format. According to Lemma \ref{lemma!G2_singular_locus} and Proposition \ref{proposition!G26_subformats}, if $Y$ is in strict $\pG_2^{(6)}$ format, then $\phi^{-1}(\sing X_{\pG_2})$ contains two components of expected dimension zero. These must therefore be supported at the vertex of $\widehat Y$, and this imposes rather strong numerical conditions on the available $\pG_2^{(6)}$ formats. Indeed, the first part of our classification algorithm (see \S\ref{sec!classification}) outputs 33 numerical $\pG_2^{(6)}$ formats for Fano 3-folds in codimension 6. In each case, we have $d(A_i)<0$ for all $i$ and $d(\lambda_{13}),d(\lambda_{24}) <0$ as well. This implies that the $\pG_2^{(6)}$ format is highly reducible.

\subsection{Fano 3-folds with $\Dih_6$ symmetry in $\pC_2$ format.}

The $\Dih_6$ invariant characters of non-negative degree are generated by $\chi^{\Dih_6}_1=\sum \chi_i + \sum \chi_{ij}$ and $\chi^{\Dih_6}_2=\sum \chi_i + 2\sum \chi_{ij}$. 
\[ \begin{array}{c|ccccc}
 & \theta_i & \theta_{ij} & A_i & A_{ij} & \lambda \\ \hline
\chi^{\Dih_6}_1 & 1 & 1 & 0 & 1 & 0 \\
\chi^{\Dih_6}_2 & 1 & 2 & 2 & 0 & 3 \\
\end{array}\]
With respect to these two characters $X_{\pC_2}$ has multigraded Hilbert series
\[ P_X(s,t) = \frac{1 - 3s^2t^2 - 3s^2t^3 - 3s^2t^4 + 2s^3t^3 + 6s^3t^4 + 6s^3t^5 + 2s^3t^6 - 3s^4t^5 - 3s^4t^6 - 3s^4t^8 + s^6t^9 }{(1-s)^3(1-t^2)^3(1-t^3)(1-st)^3(1-st^2)^3} \]
Let $X_{(a,b)}$ be the $\pC_2$ format $(X_{\pC_2},\chi_{a,b},\FF)$, where $\chi_{a,b}=a\chi^{\Dih_6}_1+b\chi^{\Dih_6}_2$. In other words $X_{(a,b)}$ is the generic regular pullback with degrees $\pC_2\left(\begin{smallmatrix}
a+2b & a+2b & a+2b \\
a & a & a \end{smallmatrix}\bigr\vert
\begin{smallmatrix}
a+b & a+b & a+b \\
2b & 2b & 2b
\end{smallmatrix}\Bigr\vert\begin{smallmatrix}3b\end{smallmatrix}\right)$.
Now 
\[ X_{(a,b)} \subset \PP^{12}\Big((a)^3,(2b)^3,3b,(a+b)^3,(a+2b)^3 \Big) \]
is a Fano 8-fold with $-K_X=\sO_X(3a + 9b)$. We can construct the following projective varieties as $\Dih_6$-invariant hyperplane sections of $X_{(a,b)}$, which therefore all carry the action of $\Dih_6$.

From the \cite{bigtables} the possible symmetric constructions are:
\begin{enumerate}
\item $X_{(0,1)}$ is a 5-fold complete intersection of codimension 4 $X_{2,2,2,3}\subset \PP^9(1^3,2^6,3)$:
\[ \theta_i\theta_j = \theta_{ij} + A_k \quad (\times 3), \qquad \theta_1\theta_2\theta_3 = A_1\theta_{23} + A_2\theta_{31} + A_3\theta_{12} + \lambda \]

\item $X_{(1,0)}$ is $\pC_2\left(\begin{smallmatrix}
1 & 1 & 1 \\
1 & 1 & 1
\end{smallmatrix}\bigr\vert
\begin{smallmatrix}
 1 & 1 & 1 \\
 0 & 0 & 0 
\end{smallmatrix}
\bigr\vert\begin{smallmatrix} 0 \end{smallmatrix}\right)$, is the Segre embedding $\PP^2\times \PP^2 \subset \PP^8$, by Proposition \ref{proposition!C2_subformats} (2.a). This gives a $\Dih_6$-symmetric construction for 
\[ \def\arraystretch{1.3}
\begin{array}{ccc}
\#41028 & Y\subset \PP(1,1,1,1,1,1,1,1) & -K_{Y}=\sO_{Y}(2) 
\end{array} \]
(Note: but \emph{not} for \#12960.)
\item $X_{(1,1)}$ is $\pC_2\left(\begin{smallmatrix}
3 & 3 & 3 \\ 1 & 1 & 1 \end{smallmatrix}\bigr\vert
\begin{smallmatrix} 2 & 2 & 2 \\ 2 & 2 & 2 
\end{smallmatrix}\bigr\vert\begin{smallmatrix} 3 \end{smallmatrix}\right)$, giving $\Dih_6$-symmetric constructions for 
\[ \def\arraystretch{1.3}
\begin{array}{ccc}
\#11222	& Y_1\subset \PP(1,1,1,2,2,3,3,3) & -K_{Y_1}=\sO_{Y_1}(1) \\
\#40407 & Y_2\subset \PP(1,1,2,2,2,3,3,3) & -K_{Y_2}=\sO_{Y_2}(2) 
\end{array} \]
\item $X_{(2,1)}$ is $\pC_2\left(\begin{smallmatrix}
4 & 4 & 4 \\
2 & 2 & 2
\end{smallmatrix}\bigr\vert
\begin{smallmatrix}
3 & 3 & 3 \\
2 & 2 & 2 
\end{smallmatrix}
\bigr\vert\begin{smallmatrix} 3 \end{smallmatrix}\right)$, giving no $\Dih_6$-symmetric constructions. (Note: we do not consider \#2511 or \#5410, since $X_{(2,1)}$ does not have variables of weight 1.)
\item $X_{(3,1)}$ is $\pC_2\left(\begin{smallmatrix}
5 & 5 & 5 \\ 
3 & 3 & 3
\end{smallmatrix}\bigr\vert
\begin{smallmatrix}
4 & 4 & 4 \\
2 & 2 & 2 
\end{smallmatrix}\bigr\vert
\begin{smallmatrix} 3 \end{smallmatrix}\right)$, giving possibly symmetric constructions for 
\[ \def\arraystretch{1.3}
\begin{array}{ccc}
\#5052 	& Y_1\subset \PP(1,1,3,4,4,5,5,5) & -K_{Y_1}=\sO_{Y_1}(1) \\ 
\#39934 & Y_2\subset \PP(1,2,3,4,4,5,5,5) & -K_{Y_2}=\sO_{Y_2}(2) \\
\end{array} \]
(Note: we do not consider \#1405, \#39678 or \#41297, since $X_{(3,1)}$ does not have two variables of weight 3.)
\end{enumerate}

\paragraph{Reid's $\ZZ/3$-Godeaux surface.} \label{note!Reids-Godeaux}
For (3), note that $X_{(1,1)}\subset \PP^{12}(1^3,2^6,3^4)$ is the regular pullback $\pC_2\left(\begin{smallmatrix}3&3&3\\1&1&1\end{smallmatrix}\bigr\vert
\begin{smallmatrix}2&2&2\\2&2&2\end{smallmatrix}\bigr\vert
\begin{smallmatrix}3\end{smallmatrix}\right)$. We get a Fano 8-fold of index 12 with Hilbert series
\[ P_{(2,3)}(t) = \frac{ 1 - 3t^4 - 3t^5 - t^6 +6t^7 + 6t^8 - t^9 - 3t^{10} - 3t^{11} + t^{15} }{ (1-t)^3(1-t^2)^6(1-t^3)^4} \]
This variety was considered by Reid in \cite{God3} Theorem 1.1. If we take $\Dih_6$-invariant hyperplane sections 
\[ \theta_1+\theta_2+\theta_3=\theta_{12}+\theta_{23}+\theta_{31}=A_1+A_2+A_3=0 \]
then we cut down to a Fano 5-fold $W\subset \PP(1^3,2^4,3^3)$ of index 5. Now we find the Fano 3-folds
\[ \def\arraystretch{1.6}
\begin{array}{ccc}
\#11222	& Y_1\subset \PP(1^3,2^2,3^3) & -K_{Y_1}=\sO_{Y_1}(1) \\
\#40407 	& Y_2\subset \PP(1^2,2^3,3^3) & -K_{Y_2}=\sO_{Y_2}(2) 
\end{array} \]
as hyperplane sections of $W$. (It is tempting to think we may also cut $W$ by hyperplanes of degrees $1$ and $3$ to construct Fano 3-fold \#8051 given by $Y_3\subset \PP(1^2,2^4,3^2)$, with $-K_{Y_3}=\sO_{Y_3}(1)$. Alas, this construction turns out to be too singular.)

Moreover, following Reid again, cutting down this second Fano 3-fold $Y_2$ by a $\Dih_6$-invariant section of degree 3 to get a surface of general type $S$ with $p_g = 2$, $K_S^2 = 3$ and an action of $\Dih_6$. Taking the quotient $S/C_3$ by the cyclic subgroup $C_3\subset \Dih_6$ gives a $\ZZ/3$-Godeaux surface with an involution. See also \cite{CU} for a detailed study of this surface using Reid's construction.

\begin{rmk}
We do not consider the case of $\Dih_8$-invariant constructions from $\pG_2$ format since, for dimensional reasons (see \S\ref{subsection!no-G2-constructions}), the full codimension 6 $\pG_2$ format does not give us any quasismooth Fano 3-folds. However it may well still be possible to obtain interesting surfaces, similar to Reid's Godeaux surface, from $\pG_2$ format.
\end{rmk}

\section{Proof of the classification}\label{sec!classification}
Let $Y\subset\PP(a)$ be a candidate Fano 3-fold from the \cite{grdb} with expected codimension 4, 5 or 6, and let $(X,\mu,\FF)$ be a cluster format as in Definition \ref{definition!cluster-format}. As explained in \S\ref{section!cluster-formats}, the character $\mu$ is determined by the choice of $\rho$ in $M^\vee\cong\ZZ^m$. We write $R$ for the polynomial ring generated by variables with weights $a_i$, such that $\PP(a)=\Proj(R)$.

\subsection{The computer search}
The algorithm proceeds in two stages. First, we search over all $\rho$ inside a certain finite polytope in $\ZZ^m$ and check the Hilbert series of the corresponding cluster format against the candidate Hilbert series. This gives a list of potential cluster formats whose numerical invariants match those of $Y$. Second, for each such numerical cluster format, we consider homogeneous maps $\phi^*\colon\sA\to R$ of degree 0. Such $\phi$ must satisfy certain further conditions in order that $Y$ be quasismooth. If these conditions are satisfied, we construct a variety $Y'$ as the projectivised regular pullback of $X$ under $\phi$, and check whether $Y'$ is really quasismooth and has the correct basket. In more detail,

\paragraph{Part 1 (Finding numerical cluster formats)} We search through all $\rho$ in $\ZZ^m$ for numerical cluster formats $(X,\mu,\FF)$ matching the Hilbert series data of the candidate $Y$.
\begin{enumerate}
\item According to \S\ref{section!cluster-formats}, the adjunction number of the cluster format $X$ is $\alpha_X=\sum_i \rho_i$. Thus we only consider $\rho$ lying on the hyperplane $(\sum\rho_i=\alpha_Y)\subset\ZZ^n$, where $\alpha_Y$ is the adjunction number of the candidate $Y$.
\item Propositions \ref{proposition!C2_subformats} and \ref{proposition!G26_subformats} determine several half-spaces in which $\rho$ must lie, in particular all $\rho_i>0$. The intersection of all these half spaces determines our finite search polytope~$P$.
\item If the cluster format is $\pG_2$ and $Y$ has codimension 4, then we assume that $Y$ is in $\pG_2^{(4)}$ subformat. According to Proposition \ref{proposition!G26_subformats}, this cuts $P$ by two further hyperplanes. For codimension 5, the $\pG_2^{(5)}$ subformat cuts $P$ by one hyperplane.
\item For each $\rho$ in $P$, we compute the Hilbert series of the corresponding cluster format and compare that with the Hilbert series of $Y$. This is computationally expensive, so we do it in two stages.
\begin{enumerate}
\item Compute equation degrees of $X$ and check whether the predicted equation degrees of $Y$ are a subset thereof.
\item For each $\rho$ satisfying (a), we compute the Hilbert numerator and compare it with that of $Y$.
\end{enumerate}
\item Each $\rho$ has an orbit under the dihedral group action, and elements of the same orbit give the same cluster format up to a coordinate change. Thus we choose a representative $\rho$ for each orbit. Sometimes there are extra symmetries. For example when $\rho$ lies on a certain facet of $P$, the cluster format specialises to $\PP^1\times\PP^1\times\PP^1$ format. Such a $\rho$ has an orbit under the octahedral group. 
\end{enumerate}
The output from Part 1 is a list of numerical cluster formats $(X,\mu,\FF)$ for the candidate $Y$.
\paragraph{Part 2 (Checking quasismoothness)} We work through necessary conditions on $\rho$ and $\phi$ imposed by the assumption that $Y$ is quasismooth.
\begin{enumerate}
\item For each potential $\rho$, we run through the \emph{reasons for failure} (see \S\ref{sec!failure}) to remove those cluster formats which it would be impossible for $\phi^{-1}(X)$ to be quasismooth.
\item We construct a general homogeneous map $\phi^*\colon \sA\to R$ of degree 0. Wherever possible, we use coordinate changes to optimise $\phi$, (see e.g.~Example \ref{eg!25-continued}).
\item We construct a test variety $Y'$, the regular pullback of $X$ along $\phi$. Depending on $\rho$, we may know a subformat for $Y'$, e.g.~$\PP^2\times\PP^2$, in which case we use this subformat to construct~$Y'$.
\item We check the quasismoothness of $Y'$. This is by far the most computationally expensive part of the algorithm. We use the strategy outlined in Section \ref{section!quasismooth-strategy}.
\item If quasismoothness fails, we try again---accidents happen! There is a chance that $Y'$ is not quasismooth for a bad random choice of $\phi$. The fact that $Y'$ is eventually quasismooth proves that our reasons for failure are sharp.
\item We check that the basket of $Y'$ matches the basket of the candidate $Y$---this is a nontrivial condition in general (see \S\ref{section!fake-baskets}).
\end{enumerate}

\subsection{Strategy for testing quasismoothness}\label{section!quasismooth-strategy}

We exploit the structure of the cluster variety to produce an efficient way of testing quasismoothness of a regular pullback $Y=\phi^{-1}(X_{\pT})$. First we compute $\bigwedge^c\left(\Jac(Y)|_{\phi^{-1}(\Pi)}\right)$, for each linear subspace $\Pi$ in the deep locus. This is fast because the Jacobian matrix is very sparse. Then we compute nonsingularity for each affine piece of the  partial covering from Lemma \ref{lem!partial-covering}. Let $F_1,\dots,F_c$ be the homogeneous equations whose restriction to $(\phi^*(\theta)\ne0)$ define the $\CI^{(c)}$ chart $Y_\theta=\phi^{-1}(X\cap(\theta\ne0))$ corresponding to cluster variable $x$. We verify the inclusion of ideals 
$\big(\phi^*(\theta)^k\big)\subset\big(\bigwedge^c\left(\Jac(F_i)\right)\big)$ for large enough $k$, which implies that the reduced singular locus of the chart $Y_\theta$ is empty.

\begin{rmk} Na\"ively checking the rank of $\Jac(Y)$ directly is not feasible in codimension $>4$ because of the size of the matrix. We have compared output of both methods in codimension~4, to ensure correct implementation.
\end{rmk}


\subsection{Reasons for failure}\label{sec!failure}
We summarise the results on Part 2 of the algorithm in the following table.

\begin{table}[ht]
\caption{Table of failures}\label{table!failure}
\renewcommand{\arraystretch}{1.1}
\setlength{\tabcolsep}{13pt}
\begin{tabular}{ccccccc}
 & Numerical & \multicolumn{4}{c}{Failures} & Working \\
Codimension & candidates & (1) & (2) & (3) & (4) & constructions \\
 \hline
4 & 1220 & 464 & 338 & 3 & 28 & 387 \\
5 & 199 & 112 & 36 & 0 & 1 & 50 \\
6 & 33 & 3 & 30 & -- & -- & 0
\end{tabular}
\end{table}
The ``Numerical candidates'' column refers to the numerical cluster constructions output from Part 1. Part 2 removes those cluster formats which fail to construct a quasismooth Fano 3-fold $Y$, and outputs those which do give a working construction. The reasons for failure listed in the table are explained in subsequent paragraphs:
\begin{enumerate}
\item $\phi^{-1}(\sing X)$ is nonempty, \S\ref{sec!fail-singular-locus};
\item $Y$ is not quasismooth at a coordinate point, \S\ref{sec!fail-quasismooth-coordinate-point};
\item $Y$ fails for some ad hoc reason \S\ref{sec!fail-misc};
\item $Y$ is quasismooth but has a false basket, \S\ref{section!fake-baskets}.
\end{enumerate}
Throughout this subsection, $\hat Y$ in $\Aa^n$ is the affine cone over a quasismooth 3-fold $Y\subset\PP(a)$ in cluster format $(X,\mu,\FF)$, and we denote the coordinates on $\Aa^n$ by $z_{1,\dots,n}$. 
The reasons for failure are conditions on the morphism $\phi\colon\Aa^n\to\Aa^N$ which are necessary for $Y$ to be quasismooth. These conditions are independent of the choice of $\phi$.

\subsubsection{Pullback singular locus of the cluster variety}\label{sec!fail-singular-locus}
The search polytope $P$ from Part (1) of the search is defined by certain numerical conditions on $\rho$ implied by the requirement that $\phi^{-1}(\sing X)$ is supported at the vertex or empty. By analysing $\phi$ more closely, we can sharpen the conditions on $\rho$.

\begin{lem}\label{lemma!divisibility}
Suppose that $\Pi$ is a component of the singular locus of $X$, with defining ideal $I_{\Pi}=(w_1,\dots,w_k)$. If $\hat Y$ is the affine cone over a quasismooth 3-fold, then one of the following two conditions must hold:
\begin{enumerate}
\item (Empty) For some $i$, $d(w_i)=0$;
\item (Vertex) For each $i=1,\dots,n$, there exists $j_i$ such that $d(z_i)$ divides $d(w_{j_i})$.
\end{enumerate}
\end{lem}
\begin{proof} (1) If $d(w_i)=0$ for some $i$, then we are done, because $\phi^*(w_i)=1$ by convention, and $\phi^{-1}(\Pi)$ is empty. (2) Otherwise, $\phi^{-1}(\Pi)$ is supported at the vertex, which implies that for each $i$, some power of $z_i$ is in $\phi^*(I_{\Pi})$. Thus for each $i$, there exists $j_i$ such that some power of $z_i$ appears in $\phi^*(w_{j_i})$.
\end{proof}

\begin{lem}\label{lemma!optimal}
Suppose we are in case (2) of Lemma \ref{lemma!divisibility}. Choose coordinates and reorder them $z_1,\dots,z_p,z_{p+1},\dots,z_n$ so that for $i=p+1,\dots,n$, we have $\phi^*(w_{j_i})=z_i$ for some $w_{j_i}$. Then at least $p$ of the $\phi^*(w_i)$ must be nontrivial modulo $(z_{p+1},\dots,z_n)$.
\end{lem}
\begin{proof}
We work on the affine subspace $V(z_{p+1},\dots,z_n)=\Aa^p\subset\Aa^n$. It remains to show that the ideal quotient $I'=\phi^*(I_{\Pi})/(z_{p+1},\dots,z_n)$ is supported on the vertex $V(z_1,\dots,z_p)$ in $\Aa^p$. For this, we need at least $\dim \Aa^p$ nontrivial generators for $I'$.
\end{proof}
\begin{eg}[Pullback of $\sing X_{\pC_2}$ is nonempty]\label{eg!25} Consider $\#25$, $Y\subset\PP(2,5,6,7,8,9,10,11)$. Part 1 of the algorithm gives five numerical $\pC_2$ formats for $Y$:
\[\left(\begin{smallmatrix}7 & 9 & 13 \\ 10 & 9 & 8 \end{smallmatrix}\big\vert
\begin{smallmatrix}7 & 11 & 10 \\ 2 & 0 & 0 \end{smallmatrix}\big\vert
\begin{smallmatrix}1\end{smallmatrix}\right),
\left(\begin{smallmatrix}8 & 9 & 13 \\ 8 & 9 & 7 \end{smallmatrix}\big\vert
\begin{smallmatrix}7 & 11 & 9 \\ 3 & 0 & 3 \end{smallmatrix}\big\vert
\begin{smallmatrix}3\end{smallmatrix}\right),
\left(\begin{smallmatrix}11 & 9 & 11 \\ 5 & 8 & 8 \end{smallmatrix}\big\vert
\begin{smallmatrix}7 & 10 & 9 \\ 6 & 0 & 4 \end{smallmatrix}\big\vert
\begin{smallmatrix}5\end{smallmatrix}\right),
\left(\begin{smallmatrix}8 & 10 & 12 \\ 8 & 7 & 9 \end{smallmatrix}\big\vert
\begin{smallmatrix}6 & 11 & 10 \\ 6 & 0 & 0 \end{smallmatrix}\big\vert
\begin{smallmatrix}3\end{smallmatrix}\right),
\left(\begin{smallmatrix}10 & 9 & 12 \\ 6 & 9 & 6 \end{smallmatrix}\big\vert
\begin{smallmatrix}8 & 10 & 8 \\ 3 & 1 & 6 \end{smallmatrix}\big\vert
\begin{smallmatrix}5\end{smallmatrix}\right)
\]
By Lemma \ref{lemma!C2_singular_locus}, the largest component of $\sing X_{\pC_2}$ is $V(\theta_1,\theta_2,\theta_3,\theta_{12},\theta_{23},\theta_{31},A_{12},A_{23},A_{31})$. We can read the degrees of the ideal generators directly from the cluster format. For example, in the first displayed case, we have $(7,9,13,7,11,10,10,9,8)$. Since none of these is divisible by $d(z_3)=6$, this case fails Lemma \ref{lemma!divisibility}. Similarly for cases 2 and 3, while case 5 also fails, because there is no space for $z_4$ which has degree $7$.
Case 4 is actually a working construction, see Example \ref{eg!25-continued} above.
\end{eg}

\begin{eg}\label{eg!166-1}
Consider $\#166$, $Y\subset\PP(2,2,3,3,4,4,5,5)$ with coordinates $p,q,r,s,t,u,v,w$ and $\pC_2$ format 
$\left(\begin{smallmatrix}5 & 5 & 5 \\ 3 & 3 & 3 \end{smallmatrix}\big\vert
\begin{smallmatrix}4 & 4 & 4 \\ 2 & 2 & 2 \end{smallmatrix}\big\vert
\begin{smallmatrix}3\end{smallmatrix}\right)$. After choosing coordinates, $\phi$ takes the form
\[\left(\begin{array}{ccc|ccc|}
v & w & P_5 & t & u & Q_4 \\
r & s & R_3 & p & q & S_2
\end{array}
\:\: T_3 \right),
\]
which passes Lemma \ref{lemma!divisibility}(2), so we test Lemma \ref{lemma!optimal}. For the component 
\[ \Pi=V(\theta_1,\theta_2,\theta_3,\theta_{12},\theta_{23},\theta_{31},A_{12},A_{23},A_{31}) \subset \sing(X_{\pC_2}),\]
the variables $r,s,t,u,v,w$ appear as pullbacks of $A_{12},A_{23},\theta_1,\theta_2,\theta_{12},\theta_{23}$ respectively. Thus we need only consider $\phi^{-1}(\Pi)\cap\Aa^2_{p,q}=V(P_5,Q_4,R_3)|_{\Aa^2_{p,q}}$. For degree reasons, $P|_{\Aa^2_{p,q}}\equiv R|_{\Aa^2_{p,q}}\equiv0$ and so $Q|_{\Aa^2_{p,q}}$ cuts out two lines. Thus $\phi^{-1}(\Pi)$ must be nonempty along the $pq$-plane.
\end{eg}

\subsubsection{Quasismoothness at coordinate points}\label{sec!fail-quasismooth-coordinate-point}
It may still happen that $Y$ is singular even though $\phi^{-1}(\sing X)=\emptyset$. The following Lemma gives a necessary condition for $Y$ to be quasismooth at all coordinate points of $\PP(a)$.
\begin{lem}
Let $\phi^*(I_X)=(\phi^*(f_1),\dots,\phi^*(f_m))$ be the ideal defining $\hat Y$ under regular pullback, and suppose that $P_i$ is the coordinate point corresponding to $z_i$.
Then for each $i$, one of the following must hold:
\begin{enumerate}
\item ($P_i\notin Y$) There exists an integer $j$ such that $\phi^*(f_j)$ contains the monomial $z_i^k$ for some $k>0$;
\item ($P_i\in Y$) There exists $S\subset\{1,\dots,n\}$ of cardinality $c=\codim_{\PP(a)} Y$ and a permutation $\sigma$ on $n$ elements, such that for all $j$ in $S$, $\phi^*(f_{j})$ contains the monomial $z_{\sigma(j)}z_i^{m_j}$ for some $m_j>0$.
\end{enumerate}
\end{lem}
\begin{proof}
Clearly, the first condition implies that $P_i$ is not contained in $Y$. So we assume that $P_i$ is in $Y$. Let $J$ denote the Jacobian matrix $\Jac(Y)$ evaluated at $P_i$. If $Y$ is quasismooth, then $J$ contains a $c\times c$ submatrix $J_c$ of rank $c$. Since $P_i$ is a coordinate point, this implies condition (2) of the Lemma. Indeed, the row numbers of $J_c$ make up the subset $S$, and $\sigma$ is some suitable permutation whose restriction to $S$ maps rows of $J_c$ to linearly independent columns of $J_c$.
\end{proof}
\begin{eg} Consider \#308 $Y\subset\PP(1,5,6,6,7,8,9,10)$ with coordinates $p,q,r,s,t,u,v,w$ in $\pC_2$ format 
$\left(\begin{smallmatrix}7 & 9 & 11 \\ 7 & 6 & 8 \end{smallmatrix}\big\vert
\begin{smallmatrix}9 & 5 & 10 \\ 0 & 6 & 0 \end{smallmatrix}\big\vert
\begin{smallmatrix}3\end{smallmatrix}\right)$. The general form of $\phi$ is
\[\left(\begin{array}{ccc|ccc|}
t & v & P_{11} & Q_9 & q & w \\
R_7 & r & u & 1 & s & 1
\end{array}
\:\: p^3 \right).
\]
Here $Y=\phi^{-1}(X)$ is always singular at the coordinate point $P_s$, even though $\phi$ is generically an immersion, and $\phi^{-1}(\sing X)$ is empty. Indeed, for degree reasons, $s$ only appears as $\phi^*(A_2)$ and possibly in $P_{11}$, $Q_9$, $R_7$. A quick examination of the equations shows that the tangent cone to $\phi^{-1}(X)$ at $P_s$ must be singular.
\end{eg}

\subsubsection{Ad hoc reasons for failure}\label{sec!fail-misc}
We document the failures appearing in column (3) of Table \ref{table!failure}
\begin{description}
\item[$(\#360)$] $Y\subset\PP^7(1, 4, 5, 6, 7, 7, 8, 9)$ in 
$\pC_2\left(\begin{smallmatrix}6 & 8 & 10 \\ 7 & 6 & 7 \end{smallmatrix}\big\vert
\begin{smallmatrix}5 & 9 & 8 \\ 4 & 0 & 0 \end{smallmatrix}\big\vert
\begin{smallmatrix}2\end{smallmatrix}\right)$ format
is always singular at one point in $\PP(4,8)$.
\item[$(\#393)$] $Y\subset\PP^7(1, 4, 5, 5, 6, 7, 8, 9)$ in 
$\pC_2\left(\begin{smallmatrix}6 & 7 & 9 \\ 6 & 8 & 8 \end{smallmatrix}\big\vert
\begin{smallmatrix}5 & 10 & 7 \\ 3 & -4 & 1 \end{smallmatrix}\big\vert
\begin{smallmatrix}0\end{smallmatrix}\right)$ format (i.e.\ $\PP^2\times\PP^2$ subformat) has a singular curve supported on $\PP(5,5)$ and a singular point in $\PP(4,8)$.
\item[$(\#878)$] $Y\subset\PP^7(1, 3, 3, 4, 4, 5, 5, 6)$ in 
$\pG_2\left(\begin{smallmatrix}9 & 6 & 9 & 6 \\ 0 & 2 & 0 & 4 \end{smallmatrix}\big\vert
\begin{smallmatrix}4 & 4 & 5 & 5 \\ 3 & 3 & 0 & 0 \end{smallmatrix}\big\vert
\begin{smallmatrix}3 \\ 3 \end{smallmatrix}\right)$ format
is singular along a curve in $\PP(3,3,6)$.
\end{description}


\subsubsection{False baskets and ice cream}\label{section!fake-baskets}
Let $Y\subset\PP(a_1,\dots,a_n)$ be a candidate Fano 3-fold from \cite{grdb} with basket of terminal quotient singularities $\sB$. Suppose $Y'$ is a quasismooth 3-fold in weighted projective space with Hilbert series matching $Y$. We say that $Y'$ has a false basket if the quotient singularities of $Y'$ are not terminal. Such $Y'$ are discarded.

\begin{eg}\label{example!false-basket} $\#569$, $Y\subset\PP(1,3,4,5,5,6,7,9)$ of Fano index $q=1$ in 
$\pG_2\left(\begin{smallmatrix}15 & 5 & 10 & 9 \\ 0 & 5 & 0 & 6 \end{smallmatrix}\big\vert
\begin{smallmatrix}7 & 3 & 7 & 8 \\ -1 & 6 & -2 & 0 \end{smallmatrix}\big\vert
\begin{smallmatrix}0 \\ 3 \end{smallmatrix}\right)$ format has $\frac15(1,1,4)$ and $\frac15(3,4,4)$ singularities instead of a single $\frac15(1,2,3)$ singularity.
\end{eg}
\begin{eg}\label{example!false-basket-2} $\#2410$, $Y\subset\PP(1,2,2,3,4,5,5,6)$ in 
$\pG_2\left(\begin{smallmatrix}6 & 5 & 9 & 6 \\ 0 & 0 & 0 & 4 \end{smallmatrix}\big\vert
\begin{smallmatrix}2 & 3 & 6 & 4 \\ 5 & 5 & -3 & 0 \end{smallmatrix}\big\vert
\begin{smallmatrix}3 \\ 4 \end{smallmatrix}\right)$ format (i.e.\ $\pA_2+\CI^{(1)}$ subformat) has a curve of index 2 singularities. This is the only quasismooth Fano 3-fold with nonisolated singularities that we find.
\end{eg}

\begin{rmk}
There is a misprint in \cite[Table 1]{bkq}: \#577 has a working construction in $\PP^2\times\PP^2$ format, while \#645 is quasismooth but not terminal.
\end{rmk}

Let $\sB$ denote either the basket of a candidate Fano 3-fold $Y$, or the set of isolated singularities on a quasismooth Fano 3-fold $Y'$. Define the \emph{basket vector} of $\sB$ to be $v(\sB)=(v_2,\dots,v_{a_n})$ where $v_r$ is the number of quotient singularities in $\sB$ with index divisible by $r$, for $2\le r\le a_n$.

\begin{eg} According to \cite{grdb}, $\#569$ has basket vector $v(\sB)=(0,3,0,1,0,0,0,1)$, while the construction of Example \ref{example!false-basket} has basket vector $v(\sB')=(0,3,0,2,0,0,0,1)$.
\end{eg}

\begin{lem} Suppose $Y$ is a candidate Fano 3-fold from \cite{grdb} in codimension 4, 5 or 6 and $Y'$ is a quasismooth 3-fold with isolated singularities such that $P_Y(t)=P_{Y'}(t)$. Then $\sB(Y)=\sB(Y')$ if and only if $v(\sB(Y))=v(\sB(Y'))$.
\end{lem}

\begin{proof}
The only if part is obvious. For the if part, we use ice cream. Define $k_Y=-q(Y)<0$ so that $\omega_Y=\sO_Y(k_Y)$. By \cite{BRZ}, the Hilbert series of $Y$ can be expressed as
\[P_Y(t)=P_I(t)+\sum_{Q\in\sB} P_{\text{orb}}(Q,k_Y)(t)\]
where $P_I(t)$ is uniquely determined by $k_Y$ and the first $\left\lfloor \frac{k_Y+2}2\right\rfloor$ terms of $P_Y(t)$.
In particular, the orbifold contributions to $P_Y(t)$ and $P_{Y'}(t)$ are equal.

Using a computer, we calculate the orbifold contribution to the Hilbert series for all baskets $\sB'$ with basket vector $v(\sB')=v(\sB(Y))$. Since $Y'$ is quasismooth with isolated singularities, one of these baskets must be the set of singularities of $Y'$. We find that the only possibilities for $\sB'$ whose orbifold contribution matches that of $Y$, are permutations of $\sB$, so the lemma is proven.
\end{proof}

We discard any constructions $Y'$ whose basket vector does not match that of the candidate $Y$, or which has nonisolated singularities. The basket vector is quite easy for the computer to determine. According to the above Lemma, the remainder are terminal quasismooth Fano 3-folds.


\noindent Stephen Coughlan \quad \texttt{stephen.coughlan@uni-bayreuth.de}

\noindent Mathematisches Institut, Lehrstuhl Mathematik VIII, Universit\"atsstrasse 30, 95447 Bayreuth, Germany

\noindent Tom Ducat \quad \texttt{t.ducat19@imperial.ac.uk}

\noindent (Previous) School of Mathematics, University of Bristol, Bristol, BS8 1TW, UK, 
and the Heilbronn Institute for Mathematical Research, Bristol, UK

\noindent (Current) Department of Mathematics, Imperial College, South Kensington Campus, 
Huxley Building, 180 Queen's Gate, London, SW7 2AZ, UK.

\end{document}